\documentclass[hidelinks,onefignum,onetabnum]{siamart251216}

\usepackage{amssymb,mathtools,mathrsfs,booktabs,color}
\usepackage{multirow,multicol}

\usepackage{enumitem}
\usepackage{graphicx,subcaption}  % 支持子图
\usepackage{array}
\usepackage{algorithm}
\usepackage{algpseudocode}

\newcommand{\brab}[1]{\bigl(#1\bigr)}
\newcommand{\braB}[1]{\Bigl(#1\Bigr)}
\newcommand{\kbrab}[1]{\bigl[#1\bigr]}
\newcommand{\kbraB}[1]{\Bigl[#1\Bigr]}

\allowdisplaybreaks[4]

\title{Energy Dissipation Analysis of Implicit-Explicit Linear Multistep Methods for Gradient Flows Using General Multipliers% \thanks{Submitted to the editors {\blue DATE}.
\thanks{Submitted.
\funding{The work of C. Quan is supported by National Natural Science Foundation of China  (Grant No. 12271241), Guangdong Basic and Applied Basic Research Foundation (Grant No. 2023B1515020030), and Shenzhen Science and Technology Innovation Program (Grant No. JCYJ20230807092402004).  The work of X. Wang is supported by National Natural Science Foundation of China (Grant No. 12501551). The work of C. Xu is supported by National Natural Science Foundation of China (Grant No. 12371408).
}}
}

\author{Chaoyu Quan\thanks{School of Science and Engineering, The Chinese University of Hong Kong (Shenzhen), Shenzhen 518172, China; Shenzhen International Center for Industrial and Applied Mathematics, Shenzhen Research Institute of Big Data, Shenzhen, 518000, China 
  (\email{quanchaoyu@cuhk.edu.cn}).}
\and Huaijin Wang\thanks{School of Mathematical Sciences and Fujian Provincial Key Laboratory of Mathematical Modeling and High Performance Scientific Computing, Xiamen University, Xiamen 361005, China
  (\email{wanghuaijin@stu.xmu.edu.cn}).}
\and Xuping Wang\thanks{School of Science and Engineering, The Chinese University of Hong Kong (Shenzhen), Shenzhen 518172, China 
  (\email{wangxuping@cuhk.edu.cn}).}
\and Chuanju Xu\thanks{School of Mathematical Sciences and Fujian Provincial Key Laboratory of Mathematical Modeling and High Performance Scientific Computing, Xiamen University, Xiamen 361005, China
  (\email{cjxu@xmu.edu.cn}).}
}

\begin{document}

\sloppy
\maketitle
\begin{abstract}
A unified framework is proposed to establish the energy dissipation of implicit-explicit linear multistep methods (IMEX-LMMs) for gradient flows, based on general multipliers that are linear combinations of first-order differences of numerical solutions. A generalized Dahlquist's theory is developed to establish the energy dissipation of IMEX-LMMs. It is shown that given an IMEX-LMM, to find a multiplier ensuring the energy dissipation is relaxed to solve a linear programming that can be easily solved. Within this framework, two specific multipliers are discovered to establish the energy dissipation of the sixth-order IMEX backward differentiation formula (IMEX-BDF6) method and a seventh-order IMEX weighted and shifted BDF method, and a new eighth-order energy-dissipative IMEX-LMM is provided. To the best of our knowledge, these are the first energy-dissipation results for the IMEX-BDF6 method and the IMEX-LMMs of order higher than six. In addition, this framework can be used directly to establish the $L^2$- or $H^1$-stability of general LMMs for linear parabolic problems. Numerical experiments illustrate the temporal accuracy and energy dissipation of these methods.
\end{abstract}

\begin{keywords}
Linear multistep method, backward differentiation formula, implicit-explicit method, gradient flow, energy dissipation
\end{keywords}

% REQUIRED
\begin{MSCcodes}
35K35, 35K55, 65M06, 65M12
\end{MSCcodes}
% 35K35  	Initial-boundary value problems for higher-order parabolic equations
% 35K55  	Nonlinear parabolic equations
% 65L06  	Multistep, Runge-Kutta and extrapolation methods for ordinary differential equations
% 65M06  	Finite difference methods for initial value and initial-boundary value problems involving PDEs
% 65M12  	Stability and convergence of numerical methods for initial value and initial-boundary value problems involving PDEs
% 65M22  	Numerical solution of discretized equations for initial value and initial-boundary value problems involving PDEs
% 65Q10  	Numerical methods for difference equations
\section{Introduction}

Gradient flows arise from many models in phase-field theory and materials science.
Consider the following gradient flow subject to periodic boundary conditions:
\begin{equation}\label{model:phase-field-eq}
    u_{t} = \mathcal{M} \kbrab{\mathcal{L}u + f(u)}, \quad u(0,\boldsymbol{x}) = u^{0}(\boldsymbol{x}) \quad \text{for } (\boldsymbol{x},t) \in \Omega\times (0,T],
\end{equation}
where $\Omega=(-\pi, \pi)^{d} \subset \mathbb{R}^{d}$ $(d = 1,2,3)$ is the domain, $T > 0$ is the final time, $u^0$ is the initial condition, $\mathcal{M}$ is a self-adjoint, linear, negative definite, and invertible operator, $\mathcal{L}$ is 
a self-adjoint, linear, and positive semi-definite operator, and $f(u)$ is a nonlinear term.
Assume that $f(u) = F'(u)$ for some potential $F$ satisfying $\int_\Omega F(u)\mathrm{d}\boldsymbol{x} \geq 0$.
Then, \eqref{model:phase-field-eq} can also be written as
\begin{equation}\label{eq:gradient-flow}
    u_{t} = \mathcal{M} \frac{\delta E}{\delta u} \quad\text{with}\quad E[u] \coloneqq \frac{1}{2} (u, \mathcal{L} u) + \int_{\Omega} F(u) \mathrm{d} \boldsymbol{x},
\end{equation}
where $(\cdot, \cdot)$ denotes the $L^{2}$ inner product defined by $(v, w) \coloneqq \int_{\Omega} v(\boldsymbol{x}) w(\boldsymbol{x}) \mathrm{d} \boldsymbol{x}$, and $\|\cdot\|$ denotes the induced $L^2$ norm.
It is well-known that the solution of a gradient flow problem satisfies the energy dissipation law
\begin{equation}\label{ieq:energy-dissipation-law}
    \frac{\mathrm{d}E}{\mathrm{d}t} = \braB{\frac{\delta E}{\delta u}, \frac{\partial u}{\partial t}} = \braB{\mathcal{M}^{-1} \frac{\partial u}{\partial t}, \frac{\partial u}{\partial t}} \leq 0,
\end{equation}
which has become a crucial criterion for designing numerical methods for gradient flows in the past decade.

Linear multistep methods (LMMs) have been widely used for time discretization of parabolic problems.
By Dahlquist's G-stability theory \cite{Dahlquist-1963-order-barrier, Dahlquist-1978-G-stability}, the order of A-stable LMMs cannot exceed two.
For higher-order LMMs, the Nevanlinna--Odeh multiplier technique \cite{1981-NO-multiplier} has been used to analyze their stability.
In particular, the stability of implicit or implicit-explicit backward differentiation formula (IMEX-BDF) methods have been well-studied up to the fifth order in \cite{2013-NO-Lubich, 2015NM-NO-AkrivisLubich,Akrivis2016BDF345, Akrivis-2018-IMA}.
For the sixth-order BDF (BDF6) method, due to the nonexistence of Nevanlinna--Odeh multiplier, a special energy technique has been introduced to establish its stability  \cite{2021SINUM-NO-Akrivis-BDF6}.
More recently, a seventh-order weighted and shifted BDF (WSBDF7) method was constructed and proved to be stable \cite{2025IMA-NO-Akrivis-BDF7}.
So far, IMEX-BDF methods have been used to solve the advection-diffusion equation \cite{AscherRuuthWetton1995}, the Navier--Stokes equation \cite{2017IMEX,2019IMEX,HuangShen-2025MCOM} and other nonlinear parabolic equations \cite{2015SINUM-NO-Akrivis,Akrivis2016,Huang-Shen-2024-SINUM,LiaoQuanTangZhou2026SemiGenerating}.

In the context of gradient flows, a fundamental concern is the energy dissipation property of numerical schemes.
Note that the energy dissipation law at the continuous level is followed by testing \eqref{model:phase-field-eq} with $u_t$. For IMEX-LMMs, a discrete analog is to test them with a simple multiplier, namely, the first-order backward difference $\delta u^{n+1}= u^{n+1}-u^n$. This multiplier yields the construction of dissipative energy up to the fifth order \cite{ChenShen1998, XuTang2006SINUM, hao2020third,LiaoKang-2024IMA-PFC}, but fails for the IMEX-BDF6 method \cite{LiaoTangZhou-2022CSIAM,QuanWangWangXuPre} despite the fact that it is stable in numerical implementations.
We mention that the energy boundedness of BDF6 can be derived using the energy technique developed in \cite{2021SINUM-NO-Akrivis-BDF6}, but not the energy dissipation. Recently, we have proved in \cite{QuanWangWangXuPre} that when testing with the simple multiplier, sixth-order energy-dissipative IMEX-LMMs can be constructed, whereas no seventh- or higher order energy-dissipative methods exist, i.e., the sixth-order barier.

In this work, instead of the simpler multiplier, we test IMEX-LMMs with the following general multiplier to establish the energy dissipation:
\begin{equation}\label{eq:multiplier1}
    \sum_{i=0}^{k-1} \mu_i \delta u^{n+1-i} \quad\text{with}\quad \sum_{i=0}^{k-1}\mu_i = 1.
\end{equation}
Since the nonlinearity of \eqref{model:phase-field-eq} is explicitly treated in IMEX-LMMs, 
a nonnegative quadratic term is required to control this explicit treatment.
As a consequence, Dahlquist's theory shall be generalized.
The contributions of this work are summarized as follows.
\begin{enumerate}
    \item[(1)] 
    Given two coprime real polynomials \(\rho(z) = \rho_0 + \rho_1 z+\ldots + \rho_r z^r\) and \( \sigma(z) = \sigma_0 + \sigma_1 z + \ldots + \sigma_r z^r\) of degree $r$, a prescribed real polynomial $\lambda(z) = \lambda_0 + \lambda_1 z + \ldots + \lambda_r z^r$, and a constant $\gamma\ge 0$, the following inequality:
    \[
        \mathrm{Re}\{\rho(z)\sigma(\bar z)\} \geq \gamma|\lambda(z)|^2 \quad\text{for any}\ z\in\mathbb{C},\ |z|\geq 1,
    \]
    is proved to be equivalent to the existence of the following quadratic decomposition:
    \begin{align}
        & 2 \biggl( \sum_{i=0}^r \rho_i v^{n+i},
            \sum_{i=0}^r \sigma_i v^{n+i} \biggr) - 2 \gamma \biggl\|\sum_{i=0}^{r} \lambda_i v^{n+i}\biggr\|^2
             \label{eq:generalized-dbc-decomp-delta-v}
        \\
         = &\; \sum_{i,j=0}^{r-1} g_{ij} (v^{n+1+i},v^{n+1+j})
        - \sum_{i,j=0}^{r-1} g_{ij} (v^{n+i},v^{n+j})
        + \biggl\|\sum_{i=0}^{r} q_i v^{n+i}\biggr\|^2 \notag
    \end{align}
    for any real sequence $\{v^i\}$, with some real positive definite matrix $G=(g_{ij})_{i,j=0}^{r-1}$ and real coefficients $q_i$. See \cref{thm:equivalence} for other equivalent statements.
    The classical Dahlquist's theory \cite{Dahlquist-1978-G-stability, BaiocchiCrouzeix1989} is consequently recovered when taking $\gamma=0$.
    
    \item[(2)] Based on the above generalized Dahlquist's theory, it is proved in \Cref{thm:main} that the energy dissipation of an IMEX-LMM for gradient flows holds if the scheme coefficients and the multiplier satisfy degree conditions, coprimality conditions, and two unit-circle positivity inequalities involving the generating polynomials of the scheme and the multiplier, and the Schur stability of the multiplier polynomial.
    These conditions form a feasibility problem \eqref{FP} over the coefficients of the scheme and the multiplier, which can be relaxed after discretization to be an easy-to-solve linear program (LP). 

    \item[(3)] Within this framework, two specific multipliers (see Proposition \ref{prop:bdf6-feasible} and \ref{prop:wsbdf7-feasible}) are discovered to establish the energy dissipation of IMEX-BDF6 and IMEX-WSBDF7 methods, and a new energy-dissipative IMEX-LMM8 is constructed.
    To the best of our knowledge, these are the first energy-dissipation results for the IMEX-BDF6 method and the IMEX-LMMs of order higher than six.
    In addition, this framework can be used directly to obtain the $L^2$- or $H^1$-stability of LMMs for linear parabolic equations (see Remark \ref{rem:linear-parabolic} for the derivation of $H^1$-stability of BDF6 method).
\end{enumerate}

The remainder of the paper is organized as follows. \Cref{sec:preliminaries} introduces the abstract setting and the formulation of IMEX-LMMs. \Cref{sec:framework} proposes a unified framework based on the general multiplier \eqref{eq:multiplier1}, establishes a generalized Dahlquist's theory, and proves modified energy dissipation together with the nonnegativity and consistency of the associated modified energy.
\Cref{sec:feasibility} formulates the feasibility problems to find multipliers to establish the modified energy dissipation of IMEX-BDF6 and IMEX-WSBDF7 and to construct the energy-dissipative IMEX-LMM8. \Cref{sec:numerics} presents numerical experiments to illustrate the temporal accuracy and modified energy dissipation of the LMMs, and \Cref{sec:conclusion} provides some concluding remarks.

\section{Preliminaries}\label{sec:preliminaries}
This section collects the preliminaries required for the subsequent analysis; see, e.g., \cite{brezis2011functional,book-solving-ODE2}.

\subsection{Abstract settings}
We introduce the abstract setting used throughout this work.
Let $H$ be a Hilbert space with inner product $(\cdot,\cdot)$ and induced norm $\|\cdot\|$. 
Assume $S$ and $V$ are Hilbert spaces such that $S \hookrightarrow V \hookrightarrow H$ with continuous and compact embeddings. Let $V^\prime$ be the dual of $V$ and denote by $\langle\cdot,\cdot\rangle_{V',V}$ the duality pairing. Identifying $H$ with its dual $H^\prime$, we have the Gelfand triple $V \subset H \subset V^\prime$.

Let $\mathcal{L}: S \to V$ be a linear, self-adjoint, and positive semi-definite operator:
\[
(\mathcal{L} u, v) = (u, \mathcal{L} v), \quad (\mathcal{L} v, v) \geq 0, \quad \forall u,v\in S.
\]
The associated semi-norm is defined by $\| \mathcal{L}^{1/2} u\| \coloneqq (\mathcal{L} u, u)^{1/2}$.
Assume that $\mathcal{M}:V\to V'$ is a linear self-adjoint operator:
\[
\langle \mathcal{M} u, v\rangle_{V',V} = \langle \mathcal{M} v, u\rangle_{V',V}, \quad \forall u, v\in V,
\]
and that the bilinear form $a_{\mathcal M}(u,v):=-\langle \mathcal M u,v\rangle_{V',V}$ 
is continuous and coercive; that is, there exist constants
$C_{\mathcal M},c_{\mathcal M}>0$ such that
\[
|\langle \mathcal{M}u,v\rangle_{V',V}|
\leq C_{\mathcal M}\|u\|_V\|v\|_V,
\quad
-\langle \mathcal{M}v,v\rangle_{V',V}
\geq c_{\mathcal M}\|v\|_V^2,
\quad \forall u, v\in V,
\]
where $\|\cdot\|_V$ denotes the norm in $V$.
Then, by the Lax--Milgram theorem, the inverse operator $\mathcal{M}^{-1}: V^\prime \to V$ 
is well-defined, linear, and bounded; moreover, 
it is self-adjoint and negative definite.
% \[
% \langle u,\mathcal{M}^{-1} v\rangle_{V^\prime,V} = \langle v,\mathcal{M}^{-1} u\rangle_{V^\prime,V}, \quad \langle v,\mathcal{M}^{-1} v\rangle_{V^\prime,V} < 0, \quad \forall u, v\in V^\prime \ \text{and} \ v\neq 0.
% \]
For $v \in H$, one can define the $(-\mathcal{M})^{-1}$-induced norm as $\|(-\mathcal{M})^{-1/2} v\| \coloneqq (v, -\mathcal{M}^{-1}v)^{1/2}$.

Assume that $f:S\to V$ is globally Lipschitz with constant $\ell_f>0$:
\begin{equation}
\label{eq:lip-f}
\|f(u)-f(v)\| \leq \ell_f \|u-v\|,\quad \forall u,v \in S.
\end{equation}
This implies that the potential function $F$ satisfies the following inequality:
\begin{equation}
\label{eq:lip-F}
|(F(u) - F(v), 1) - (f(v), u-v)| \leq \frac{\ell_f}{2} \| u-v \|^2,\quad \forall u,v\in S.
\end{equation}
\begin{remark}
The Lipschitz continuity assumption for the nonlinear term is widely used in the analysis of numerical methods for gradient flows. For nonlinearities that do not satisfy this condition globally, a common remedy in numerical simulations is to use truncation techniques; see, e.g. \cite{ShenYang-2010-Lipschitz,LiQiaoWang-2021MCOM-truncation,FuTangYang-2024IERK}.
\end{remark}
%A rigorous analysis of gradient flows without the global Lipschitz assumption will be pursued in our future work.

Furthermore, we assume that there exist two constants $\zeta > 0$ and $0 < \eta \leq 1$ such that the following inequality holds:
\begin{equation}
\label{eq:l2-control}
\|v\| \leq \zeta \| (-\mathcal{M})^{-1/2} v\|^\eta  \| \mathcal{L}^{1/2} v \|^{1-\eta},\quad \forall v\in S.
\end{equation}

We list several typical gradient flows with parameter $\varepsilon>0$ to illustrate the abstract setting:
% to better understand the above notations:
\begin{itemize}
    \item Allen--Cahn equation: $\mathcal{M}=-\mathcal{I}$, $\mathcal{L} = -\varepsilon^2 \Delta$,  $f(u) = u^3-u$, $H=L^2(\Omega)$, $V=L^2(\Omega)$, and $S = H^2(\Omega)$. Inequality \eqref{eq:l2-control} holds for $\zeta=1$ and $\eta=1$.
    \item Cahn--Hilliard equation: $\mathcal{M} = \Delta$, $\mathcal{L} = -\varepsilon^2 \Delta$, $f(u) = u^3-u$, $H=\dot{L}^2(\Omega)$, $V=\dot{H}^1(\Omega)$, and $S=\dot{H}^3(\Omega)$. Inequality \eqref{eq:l2-control} holds for $\zeta=\varepsilon^{-1/2}$ and $\eta=1/2$.
    \item Phase field crystal (PFC) equation: $\mathcal{M} = \Delta$, $\mathcal{L} = (\mathcal{I}+ \Delta)^2 + \varepsilon \mathcal{I}$, $f(u) = u^3-2\varepsilon u$, $H=\dot{L}^2(\Omega)$, $V=\dot{H}^1(\Omega)$, and $S=\dot{H}^5(\Omega)$. Inequality \eqref{eq:l2-control} holds for $\zeta= (2\sqrt{1+\varepsilon}-2)^{-1/4}$ and $\eta=1/2$.
\end{itemize}
Here, all Sobolev spaces are understood in the periodic sense, consistent with the periodic boundary conditions.
For a Hilbert space $X$, $\dot{X} = \{v\in X: \int_{\Omega} v \mathrm{d} \boldsymbol{x} = 0\}$ denotes the subspace of functions with zero mean, which is standard for gradient flows preserving mass.
% Here, the Sobolev spaces are actually periodic Sobolev spaces as the periodic boundary condition is equipped with, and for a Hilbert space $X$, $\dot{X} = \{v\in X: \int_{\Omega} v \mathrm{d} \boldsymbol{x} = 0\}$ denotes the subspace of functions with zero mean, which is a common requirement for mass-conserving gradient flows.

% If $\mathcal{L}$ or $\mathcal{M}$ are only PSD, definiteness can typically be recovered by restricting the operators to suitable subspaces (e.g., the zero-mean space in the Cahn--Hilliard equation) or by introducing appropriate stabilization techniques (e.g., in the phase field crystal equation).

\subsection{Formulation of IMEX-LMMs}
Let $t_n = n\tau$, $n=0,\ldots,N$, be a uniform partition of $[0,T]$ with step size $\tau = T/N$. Given initial approximations $u^i\in S$ for $0\le i\le k-1$,  the $k$-step IMEX-LMM for \eqref{model:phase-field-eq} reads
\begin{equation}\label{eq:model-lmm2}
    \sum_{i=0}^k A_i^{(k)} u^{n+1-i} = \tau \mathcal{M} \kbraB{\sum_{i=0}^k B_i^{(k)} \mathcal{L} u^{n+1-i} + \sum_{i=1}^k \hat{B}_i^{(k)} f(u^{n+1-i})}
\end{equation}
for $k-1\leq n\leq N-1$, where the coefficients satisfy the $k$th-order conditions:
\begin{equation}
\label{cond:order-lmm}
\left\{
\begin{aligned}
& \sum_{i=0}^k A_i^{(k)} = 0, \\
& \sum_{i=0}^k A_i^{(k)} (-i)^{m+1} = (m+1) \sum_{i=0}^k B_i^{(k)} (-i)^{m}, \quad 0\leq m\leq k-1,\\
& \sum_{i=0}^k A_i^{(k)} (-i)^{m+1} = (m+1) \sum_{i=1}^k \hat{B}_i^{(k)} (-i)^{m}, \quad 0\leq m\leq k-1.\\
\end{aligned}
\right .
\end{equation}
This IMEX-LMM treats the linear part $\mathcal L u$ implicitly and the nonlinear part $f(u)$ explicitly, avoiding nonlinear systems at each time step. 
To simplify the analysis, we further impose the following normalization condition:
\begin{equation}
\label{cond:normalization}
 \sum_{i=0}^k A_i^{(k)} (-i) =  \sum_{i=0}^k B_i^{(k)} = \sum_{i=1}^k \hat{B}_i^{(k)} = 1,
\end{equation}
which ensures that $\frac{1}{\tau} \sum_{i=0}^{k} A_i^{(k)} u(t_{n+1-i})$ consistently approximates $u_t(t_{n+1})$.
 Applying $(\tau \mathcal M)^{-1}$ to \eqref{eq:model-lmm2}, we obtain the equivalent reformulation
 \begin{equation}\label{scheme:lmm-diff-form}
\begin{aligned}
&\frac{1}{\tau} \sum_{i=0}^{k-1} a_{i}^{(k)} \mathcal{M}^{-1}\brab{\delta u^{n+1-i}}
- \sum_{i=0}^{k} B_{i}^{(k)} \mathcal{L}  u^{n+1-i} 
= f(u^n) + \sum_{i=1}^{k-1} \hat{b}_i^{(k)} \delta f(u^{n+1-i})
\end{aligned}
\end{equation}
for $k-1\leq n\leq N-1$, 
where the reformulated coefficients are given by
\begin{equation}\label{def:coefs-a-b-c}
    a_i^{(k)} = \sum_{j=0}^i A_j^{(k)} \ \text{for } 0\leq i\leq k-1 \quad\text{and}\quad \hat{b}_i^{(k)} = \sum_{j=1}^i \hat{B}_j^{(k)} - 1 \ \text{for }1\leq i\leq k-1.
\end{equation}
% \begin{equation}
% \label{def:coefs-a-b-c}
% \left\{
% \begin{aligned}
% & a_i^{(k)} = \sum_{j=0}^i A_j^{(k)},&& 0\leq i\leq k-1 ,\\
% & \hat{b}_i^{(k)} = \sum_{j=1}^i \hat{B}_j^{(k)} - 1 ,&& 1\leq i\leq k-1.
% \end{aligned}
% \right .
% \end{equation}
% For simplicity, we define the coefficient vectors:
For ease of notation, we introduce the coefficient vectors
\begin{equation}
\label{eq:coef-vecs}
\begin{aligned}
    \boldsymbol{A}^{(k)} &= [A_0^{(k)},\ldots,A_k^{(k)}]^\top,\;
    \boldsymbol{B}^{(k)}=[B_0^{(k)},\ldots,B_k^{(k)}]^\top,\;
    \hat{\boldsymbol{B}}^{(k)}=[\hat{B}_1^{(k)},\ldots,\hat{B}_k^{(k)}]^\top, \\
    \boldsymbol{a}^{(k)} &= [a_0^{(k)},\ldots,a_{k-1}^{(k)}]^\top,\  
    \hat{\boldsymbol{b}}^{(k)}=[\hat{b}_1^{(k)},\ldots,\hat{b}_{k-1}^{(k)},0]^\top.
\end{aligned}
\end{equation}

\section{Modified energy dissipation of IMEX-LMMs}\label{sec:framework}

This section develops a framework for preserving the modified energy dissipation of IMEX-LMMs.
% This section develops a framework for the modified energy dissipation preservation of IMEX-LMMs.
We start by introducing the notation used throughout this section
and defining the modified discrete energy $E_G^n$ in the form of \eqref{eq:EG}. We then
establish a generalized Dahlquist's theory in
\cref{subsec:general-baio-crou} to derive the positive-definiteness conditions for energy dissipation in \cref{subsec:properties-mod-energy}. We further show in \cref{subsec:properties-mod-energy} the consistency of $E_G^n$ with the original energy together with sufficient conditions for its nonnegativity. These results are then  combined in \cref{thm:main} to prove the modified energy dissipation of IMEX-LMMs.

We  rewrite \eqref{eq:multiplier1} as
\begin{equation}
\label{eq:multiplier2}
 \sum_{i=0}^{k-1} \mu_i \delta u^{n+1-i} = \sum_{i=0}^k \nu_i u^{n+1-i},
 \end{equation}
where $\nu_0=\mu_0$, $\nu_k=-\mu_{k-1}$, and $\nu_i=\mu_i-\mu_{i-1}$ for $1 \leq i \leq k-1$. 
We denote the corresponding coefficient vectors by
\begin{equation}
\label{eq:mu-nu}
\boldsymbol{\mu} = [\mu_0,\ldots,\mu_{k-1}]^\top\quad \text{and}\quad 
\boldsymbol{\nu} = [\nu_0,\ldots,\nu_{k}]^\top.
\end{equation}

For a given vector $\boldsymbol{s}=[s_0,\ldots,s_r]^\top$, we define the generating polynomial as
\begin{equation}\label{eq:genpoly}
M(z;\boldsymbol{s})\coloneqq \sum_{i=0}^r s_i z^{r-i}.
\end{equation}
Here, $\boldsymbol{s}$ can stand for  $\boldsymbol{a}^{(k)}$, $\boldsymbol{B}^{(k)}$, $\boldsymbol{\mu}$, or $\boldsymbol{\nu}$. 
By \eqref{eq:multiplier2}, we have $
M(z;\boldsymbol{\nu}) = (z-1)M(z;\boldsymbol{\mu})$.

Throughout this work, for two vectors of  functions $\boldsymbol{w} = [w^1,\ldots,w^m]^\top$ and $\boldsymbol{v}=[v^1,\ldots,v^m]^\top$, with $w^i,v^i\in H$, we define the inner product by
\[
(\boldsymbol{w},\boldsymbol{v})
\coloneqq \sum_{i=1}^m (w^i,v^i).
\]
More generally, for a matrix $G=(g_{ij})\in\mathbb R^{m\times m}$, we define the $G$-weighted bilinear form by
\[
(\boldsymbol{w}, \boldsymbol{v})_G\coloneqq \sum_{i,j=1}^m g_{ij}(w^i,v^j).
\]
When $G$ is positive definite, $(\cdot,\cdot)_G$ defines an inner
product on the product space. If $\mathscr{T}$ is a linear operator, typically $\mathcal{L}$ or
$\mathcal{M}^{-1}$, then $\mathscr{T}\boldsymbol v$ is understood componentwise, namely $\mathscr{T}\boldsymbol v
=[\mathscr{T}v^1,\ldots,\mathscr{T}v^m]^\top$.

In addition, we define the following modified energy:
\begin{equation}
\label{eq:EG}
E_G^n = 
 -\frac{1}{\tau} (\boldsymbol{v}_n,\mathcal{M}^{-1} \boldsymbol{v}_{n})_{G_a}
 + (\boldsymbol{u}_{n}, \mathcal{L} \boldsymbol{u}_{n})_{G_B}
     + \sum_{i=0}^{k-1} \mu_i (F(u^{n-i}) ,  1) 
 +  \frac{\ell_f}{2}   \sum_{i=1}^{k-1}  \hat{c}_i \| \delta u^{n+1-i}\|^2 
\end{equation}
for two given positive definite matrices $G_a \in \mathbb{R}^{(k-1)\times (k-1)}$ and $G_B \in \mathbb{R}^{k \times k}$. 
Here, 
\begin{equation}
\label{eq:eng-vn-un}
    \boldsymbol{v}_{n} = [\delta u^{n}, \delta u^{n-1}, \dots, \delta u^{n+2-k}]^\top,
    \quad
    \boldsymbol{u}_{n} = [ u^{n},  u^{n-1}, \dots,  u^{n+1-k}]^\top,
\end{equation}
$\ell_f$ is the Lipschitz constant in \eqref{eq:lip-f}, and $\hat{c}_i = \sum_{j=i}^{k-1} \tilde{c}_j$ for $0\leq i \leq k-1$, where
% the product of two vector functions is defined by $\boldsymbol{u}^\top \boldsymbol{v} = \sum_i (u^i,v^i)$, and $\hat{c}_i = \sum_{j=i}^{k-1} \tilde{c}_j$ for $0\leq i \leq k-1$, where
\begin{equation}
\label{eq:tilde-c}
\tilde{c}_i  = 
(1- \delta_{i,0}) | \hat{b}_i^{(k)}|  \sum_{j=0}^{k-1}   |\mu_j|   +  |\mu_i| \sum_{j=1}^{k-1} | \hat{b}_j^{(k)}| 
 + |\mu_i| + \Bigl ( |\mu_i| + \sum_{j=i}^{k-1} j|\mu_j| \Bigr ) (1-\delta_{i,0}).
\end{equation} 
Since $G_a$ and $G_B$ are positive definite, there exist constants $\lambda_a,\lambda_B>0$ such that
\begin{equation}
\label{eq:lambda-ab}
\boldsymbol{x}^\top G_a \boldsymbol{x} \geq \lambda_a \boldsymbol{x}^\top \boldsymbol{x},\quad
\boldsymbol{y}^\top G_B \boldsymbol{y} \geq \lambda_B \boldsymbol{y}^\top \boldsymbol{y},\quad
\forall \boldsymbol{x}\in\mathbb{R}^{k-1}, \ \boldsymbol{y}\in\mathbb{R}^{k}.
\end{equation}

The interpretation of $E_G^n$ as a modified energy will be justified by the results in \cref{subsec:properties-mod-energy}. For matrices $G_a$ and $G_B$ that meet the positive-definiteness conditions in \cref{lemma:multiplier-LMM-energy}, the dissipation property of $E_G^n$ follows. The non-negativity of $E_G^n$ requires the additional assumptions of \cref{lem:1}, while its consistency with the original energy follows from \cref{rem:consistency} using the property of $G_B$ stated in \cref{lem:1gnu1}.

% \textcolor{magenta}{The modified energy \eqref{eq:EG} is an $O(\tau)$ approximation of the original energy $E[u^n]$ in \eqref{eq:gradient-flow}; see \cref{lem:1gnu1,rem:consistency}. Under the additional assumptions of \cref{lem:1}, $E_G^n$ is uniformly nonnegative.}

\subsection{A generalized Dahlquist's theory}
\label{subsec:general-baio-crou}
In this subsection, we establish a generalized Dahlquist's theory; see classical Dahlquist's G-stability theory in \cite{Dahlquist-1978-G-stability, BaiocchiCrouzeix1989} or \cite[Section V.6]{book-solving-ODE2}.
% \textcolor{red}{In this subsection, we establish a generalized \textcolor{red}{Dahlquist} equivalence
% theorem, extending the classical algebraic result of Baiocchi and Crouzeix \cite{BaiocchiCrouzeix1989}; see also \cite{Dahlquist-1978-G-stability} and \cite[Section V.6]{book-solving-ODE2}.}
The extension consists of introducing the additional nonnegative term $\gamma|\lambda(z)|^2$, which is crucial for the construction of positive definite matrices in \cref{lemma:multiplier-LMM-energy} and therefore for the modified energy dissipation analysis of the IMEX-LMMs in \cref{thm:main}.

% Let
% \[
% \rho(z) = \rho_0 + \rho_1 z+\ldots + \rho_r z^r\quad  \text{and}\quad \sigma(z) = \sigma_0 + \sigma_1 z + \ldots + \sigma_r z^r
% \]
% be real polynomials of  degree $r$, where $r\ge 1$ and $\rho_r\sigma_r\neq 0$.
% Let $\lambda(z)=\lambda_0+\lambda_1z+\cdots+\lambda_r z^r$ be a nonzero real polynomial of degree at most $r$.
% For a polynomial $p(z)=\sum_{j=0}^r p_j z^j$ of degree $r$, we define its coefficient vector by $\boldsymbol p=[p_r,\ldots,p_0]^\top$ and denote its degree by $\deg p$.
% In particular, $\boldsymbol\rho$, $\boldsymbol\sigma$, and $\boldsymbol\lambda$ denote the
% coefficient vectors of $\rho$, $\sigma$, and $\lambda$, respectively.

In this subsection, for any polynomial
$\pi(z)=\sum_{j=0}^r \pi_j z^j$, we denote $\boldsymbol{\pi}=[\pi_r,\ldots,\pi_0]^\top$ 
as its coefficient vector. Henceforth, $\boldsymbol{0}_{m}$ denotes the zero vector in $\mathbb{R}^m$.
% $\boldsymbol{0}_{m\times n}$ denotes the zero matrix in $\mathbb{R}^{m\times n}$.
% ; whenever the dimension is clear from the context, the subscripts are omitted.

\begin{theorem}\label{thm:equivalence}
Let
\[
\rho(z) = \rho_0 + \rho_1 z+\ldots + \rho_r z^r\quad  \text{and}\quad \sigma(z) = \sigma_0 + \sigma_1 z + \ldots + \sigma_r z^r
\]
be two relatively prime real polynomials of degree $r \ge 1$ with $\rho_r\sigma_r\neq 0$.  
Given a nonzero real polynomial $\lambda(z)=\lambda_0+\lambda_1z+\cdots+\lambda_r z^r$ of degree at most $r$ and a constant $\gamma \geq 0$, the following statements are equivalent:
\begin{enumerate}[label=(\roman*),font=\upshape]
\item 
\begin{equation}
\label{eq:equi-thm:cond1}
 \operatorname{Re} \Bigl\{ \frac{\rho(z)}{\sigma(z)}\Bigr \} > \gamma \Bigl| \frac{ \lambda(z) }{\sigma(z)} \Bigr|^2,\quad \forall |z|>1 .
\end{equation}

\item
\begin{equation}
\label{eq:equi-thm:cond2}
\operatorname{Re}\bigl\{ \rho(z) \sigma(\bar{z})\bigr\} \geq \gamma | \lambda(z) |^2,\quad \forall |z| \geq 1.
\end{equation}

\item 
\begin{enumerate}[label=(\alph*),font=\upshape]
\item $\sigma(z) \neq 0,\quad \forall |z|>1$,
\item $\operatorname{Re} \{  \rho(z) \sigma(\bar{z}) \} \geq \gamma |\lambda(z)|^2,\quad \forall |z|=1$,
\item If  $z_0$ is a root of $\sigma(z)$ with $|z_0|=1$, then $z_0$ is a simple root, and $\frac{\rho(z_0)}{\sigma'(z_0)} \bar{z}_0 > 0$.
\end{enumerate}

\item There exist $r$ linearly independent  polynomials $p_1,\ldots,p_r$, and $q$ with real coefficients, $\deg p_j \leq r-1$, and $\deg q \leq r$, such that for any $z,w\in\mathbb{C}$,
\begin{equation}
\label{eq:equi-thm:cond4}
\rho(z) \sigma(w) + \rho(w) \sigma(z) - 2\gamma  \lambda(z) \lambda(w)  = q(z) q(w) + (zw-1) \sum_{j=1}^r p_j(z) p_j(w).
\end{equation}

\item
There exist $\boldsymbol{q} = [q_r,\ldots,q_0]^\top \in\mathbb{R}^{r+1}$ and a positive definite matrix $G\in \mathbb{R}^{r\times r}$ such that
\begin{equation}
\label{eq:equi-thm:cond5}
\boldsymbol{\rho}\boldsymbol{\sigma}^\top
+
\boldsymbol{\sigma} \boldsymbol{\rho}^\top
-2\gamma  \boldsymbol{\lambda} \boldsymbol{\lambda}^\top
=
\boldsymbol{q}\boldsymbol{q}^\top
+
\begin{bmatrix}
G & \boldsymbol{0}_{r} \\
\boldsymbol{0}_{r}^{\top} & 0
\end{bmatrix}
-
\begin{bmatrix}
0 & \boldsymbol{0}_{r}^{\top} \\
\boldsymbol{0}_{r} & G
\end{bmatrix},
\end{equation}
where $\boldsymbol{\rho}$, $\boldsymbol{\sigma}$, and
$\boldsymbol{\lambda}$ are the coefficient vectors of $\rho$, $\sigma$, and $\lambda$, respectively.
\end{enumerate}

\end{theorem}

\begin{proof}
We divide the proof into three parts: $\text{(i)} \Rightarrow \text{(ii)} \Rightarrow \text{(iii)} \Rightarrow \text{(i)}$, followed by $\text{(ii)} \Leftrightarrow \text{(iv)}$, and finally $\text{(iv)} \Leftrightarrow \text{(v)}$.

\paragraph{Part A: Equivalence of {\rm (i)}--{\rm (iii)}} 
% ~\\
\noindent \textrm{(i) $\Rightarrow$ (ii):} Assume that (i) holds.
First, we show that $\sigma(z) \neq 0$ for all $|z| > 1$. Suppose by contradiction that $\sigma(z)$ has a root $z^*$ with $|z^*| > 1$ of multiplicity $m \ge 1$. Since $\rho(z)$ and $\sigma(z)$ are relatively prime, $\rho(z^*) \neq 0$. Thus, $z^*$ is a pole of $\rho(z)/\sigma(z)$ of order $m$. In a neighborhood of $z^*$, the Laurent expansion gives 
\begin{equation}
\label{eq:equi-thm:1}
\frac{\rho(z)}{\sigma(z)} = a_{-m}(z-z^*)^{-m} + O((z-z^*)^{-m+1}) \quad \text{with}\quad  a_{-m} \neq 0.
\end{equation}
Let $a_{-m}=|a_{-m}|\mathrm e^{\mathrm{i} \phi}$, where
$\phi=\arg(a_{-m})$ denotes the phase angle of $a_{-m}$.
We approach $z^*$ along the ray $z = z^* + \varepsilon \mathrm{e}^{\mathrm{i}\theta}$ for $\varepsilon > 0$. The real part of the leading term in \eqref{eq:equi-thm:1} is $|a_{-m}| \varepsilon^{-m} \cos(\phi - m\theta)$. Since $m \ge 1$, we can choose an angle $\theta^*$ such that $\cos(\phi - m\theta^*) = -1$. Along the ray $\theta=\theta^*$, as $\varepsilon \to 0^+$, this real part approaches $-\infty$, dominating the remaining terms. Then there exists $\varepsilon^*>0$ such that
\begin{equation}
\label{eq:equi-thm:10}
 \operatorname{Re}\left\{ \frac{\rho(z)}{\sigma(z)} \right\} < 0,\quad \forall z = z^*+\varepsilon \mathrm{e}^{\mathrm{i}\theta^*}\quad \text{with}\quad 0<\varepsilon<\varepsilon^*.
\end{equation}
Choosing $\varepsilon$ small enough to ensure $|z|>1$ contradicts condition (i). Thus, $\sigma(z) \neq 0$ for all $|z| > 1$. 

Since $|\sigma(z)|^2 > 0$ for $|z|>1$, we multiply both sides of \eqref{eq:equi-thm:cond1} by $|\sigma(z)|^2 = \sigma(z)\sigma(\bar{z})$ to obtain $\operatorname{Re}\{\rho(z)\sigma(\bar{z})\} > \gamma | \lambda(z) |^{2}$ for all $|z| > 1$. 
Since the function $z \mapsto \operatorname{Re}\{\rho(z)\sigma(\bar z)\}-\gamma |\lambda(z)|^2$
is continuous in $\mathbb{C}$, letting $|z|\to 1$ yields condition (ii).

\textrm{(ii) $ \Rightarrow $ (iii):} Assume that (ii) is true. Condition (iii)(b) is obtained by restricting (ii) to $|z|=1$. For condition (iii)(a), we suppose by contradiction that $z^*$ with $|z^*|>1$ is a root of $\sigma(z)$ of multiplicity $m\geq 1$. By the assumption of relative primality, $\rho(z^*)\neq 0$. Then by the proof of $\text{(i)}\Rightarrow \text{(ii)}$, there exist $\varepsilon^*>0$ and $\theta^*$ such that \eqref{eq:equi-thm:10} holds. Thus, we have
\[
\operatorname{Re}\{\rho(z)\sigma(\bar{z})\}  = |\sigma(z)|^2 \operatorname{Re}\left\{\frac{\rho(z)}{ \sigma(z)}\right\} < 0,\quad \forall z = z^*+\varepsilon \mathrm{e}^{\mathrm{i}\theta^*}\quad \text{with}\quad 0<\varepsilon <\varepsilon^*.
\]
By choosing $\varepsilon$ small enough so that
$|z|\geq1$, we obtain $\operatorname{Re}\{\rho(z)\sigma(\bar{z})\} < 0$. This contradicts condition (ii).

For condition (iii)(c), assume that  $z_0 = \mathrm{e}^{\mathrm{i}\theta_0}$ is a root of $\sigma(z)$ of multiplicity $m\geq 1$. 
Then $z_0$ is a pole of $\rho(z)/\sigma(z)$ of order $m$.
In a neighborhood of $z_0$, we have the Laurent expansion \eqref{eq:equi-thm:1}, with $z^*$ replaced by $z_0$.
When $z$ approaches $z_0$ from the exterior $|z|>1$ along the 
trajectory $z = z_0(1+\varepsilon \mathrm{e}^{\mathrm{i}\theta})$ for $\varepsilon >0$ and $\theta \in [-\pi/2, \pi/2]$, the real part of the leading term of the Laurent series of $\rho(z)/\sigma(z)$ is
\begin{equation}
\label{eq:equi-thm:9}
 |a_{-m}| \varepsilon^{-m} \cos(\arg(a_{-m}\bar{z}_0^m) - m\theta).
\end{equation}
By condition (ii) and $\operatorname{Re}\{\rho(z)\sigma(\bar{z})\} = |\sigma(z)|^2 \operatorname{Re} \{\rho(z)/\sigma(z)\}$, \eqref{eq:equi-thm:9} must be nonnegative for all $\theta \in [-\pi/2, \pi/2]$ as $\varepsilon \to 0^+$. If $m \ge 2$, the angle $m\theta$ spans an interval of length of at least  $2\pi$. This makes \eqref{eq:equi-thm:9} strictly negative for some $\theta$, contradicting (ii). Thus, the root must be simple, i.e., $m=1$ and hence $\sigma'(z_0)\neq 0$. For $m=1$, $\cos(\arg(a_{-1}\bar{z}_0) - \theta)$ must remain nonnegative for all $\theta \in [-\pi/2, \pi/2]$, which requires that $\arg(a_{-1}\bar{z}_0)$ is an integer multiple of $2\pi$. Then $a_{-1}\bar{z}_0$ is a positive real number. Since the residue of $\rho(z)/\sigma(z)$ at $z_0$ is $a_{-1} = \rho(z_0)/\sigma'(z_0)$, this implies $\frac{\rho(z_0)}{\sigma'(z_0)}\bar{z}_0 > 0$.

\textrm{(iii) $\Rightarrow$ (i):} Assume that (iii) is true. 
Let $z_1, \dots, z_s$ be the simple boundary roots of $\sigma(z)$ on $|z|=1$. By condition (iii)(c), $c_j \coloneqq \bar{z}_j\rho(z_j)/\sigma'(z_j) > 0$. We isolate the boundary singularities by defining 
\begin{equation}
\label{eq:equi-thm:15}
h(z) \coloneqq \frac{\rho(z)}{\sigma(z)} - \sum_{j=1}^s \frac{c_j z_j}{z-z_j},
\end{equation}
making $h(z)$ holomorphic in $|z|>1$ and continuous for $|z| \ge 1$. Using the identity $\operatorname{Re}\{z_j/(z-z_j)\} = \frac{1}{2}(|z|^2-1)/|z-z_j|^2 - \frac{1}{2}$ for $z\neq z_j$, we have 
\begin{equation}
\label{eq:equi-thm:13}
 \operatorname{Re}\left \{\frac{\rho(z)}{\sigma(z)} \right \}   =  \operatorname{Re}\{h(z)\} - \sum_{j=1}^s \frac{c_j}{2}  + \sum_{j=1}^s \frac{c_j}{2} \frac{|z|^2-1}{|z-z_j|^2}
.
\end{equation}
Let 
\begin{equation}
\label{eq:equi-thm:14}
V(z) \coloneqq  \operatorname{Re}\left \{\frac{\rho(z)}{\sigma(z)} \right \}  - \gamma \left| \frac{\lambda(z)}{\sigma(z)}\right|^2= H(z) +\sum_{j=1}^s \frac{c_j}{2} \frac{|z|^2-1}{|z-z_j|^2},
\end{equation}
where 
\[
H(z) \coloneqq \operatorname{Re}\{h(z)\} - \sum_{j=1}^s \frac{c_j}{2}  - \gamma \Bigl| \frac{\lambda(z)}{\sigma(z)} \Bigr|^2 .
\]
Then $H(z)$ is continuous in $|z|>1$. 
Since $\sigma$ has real coefficients, $\bar z_j$ is also a root of $\sigma$.
If $\gamma>0$, then condition (iii)(b) evaluated at $z=z_j$ gives 
\[0=\operatorname{Re}\{\rho(z_j)\sigma(\bar z_j)\}\ge \gamma |\lambda(z_j)|^2.
\]
Hence $\lambda(z_j)=0$.
Since $z_j$ is a simple root of $\sigma$ by condition (iii)(c), the quotient
$\lambda(z)/\sigma(z)$ has a removable singularity at $z_j$.
Therefore, $H$ extends continuously to $|z|=1$.
If $\gamma=0$, the last term in $H$ is absent, and then the continuity of $H$ is immediately obtained. Thus, by condition (iii)(b) and \eqref{eq:equi-thm:14}, we have $H(z)\geq 0$ for $|z|=1$.

Using the map $w=1/z$, which maps $|z|\geq 1$ to $|w|\leq 1$, we define
\[
\begin{aligned}
&\tilde{\rho}(w) \coloneqq w^r\rho(1/w),\quad
\tilde{\sigma}(w) \coloneqq w^r\sigma(1/w),\quad
\tilde{\lambda}(w) \coloneqq w^r\lambda(1/w),\\
&
\tilde{h}(w) \coloneqq h(1/w),\quad
\tilde{H}(w) \coloneqq H(1/w).
\end{aligned}
\] 
Then $\tilde{\rho}$, $\tilde{\sigma}$, and $\tilde{\lambda}$ are polynomials, with
$\tilde{\sigma}(0)=\sigma_r\neq 0$. By the preceding discussion, $\tilde{h}$ is holomorphic
in $|w|<1$ and continuous for $|w|\leq1$, $\tilde{\lambda}/\tilde{\sigma}$ is holomorphic in $|w|<1$ and continuous on $|w|=1$ when $\gamma>0$, $\tilde{H}(w)$ is continuous for $|w|\leq 1$, and $\tilde{H}(w)\geq 0$ on $|w|=1$.

Since $\tilde{h}(w)$ is holomorphic in $|w|<1$ and continuous for $|w|\leq 1$, $\operatorname{Re}\{\tilde{h}(w)\}$ is harmonic for $|w|<1$ and continuous for $|w|\leq 1$.
Using the identity of the  Laplacian for a  holomorphic function $g$: $\Delta|g|^2 = 4|g'|^2$, we find 
\[
\Delta \tilde{H}(w) = -4\gamma |(\tilde{\lambda}(w)/\tilde{\sigma}(w))'|^2 \le 0,
\]
meaning $\tilde{H}(w)$ is a superharmonic function in $|w|<1$ and continuous for $|w|\leq 1$.
Applying the minimum principle for the superharmonic function  $\tilde{H}(w)$ gives $\tilde{H}(w) \geq 0$ for $|w| < 1$. Thus, $H(z)\geq 0 $ for $|z|>1$. In the following, we show that condition (i) is met, i.e., $V(z)>0$ for $|z|>1$, in two cases.

\begin{enumerate}[label=\arabic*)]
    \item  In the case of $s\geq 1$, from  $c_j > 0$ and \eqref{eq:equi-thm:14}, we have
    \[
V(z) = H(z) + \sum_{j=1}^s \frac{c_j}{2} \frac{|z|^2-1}{|z-z_j|^2} > 0\quad \text{for}\quad |z|>1.
\]

\item In the case of $s=0$, we have $V(z) = H(z) \geq 0$ for $|z|>1$. We show that $V(z)>0 $ for $|z|>1$ by contradiction. Suppose $V(z_*)= 0$ for some $|z_*|>1$. Then $\tilde{H}(w_*) = 0$ where $w_*=1/z_*$ and $|w_*|<1$. The strong minimum principle for the superharmonic function  $\tilde{H}(w)$   implies $\tilde{H}(w)\equiv 0$ for $|w|<1$, which is equivalent to $V(z)\equiv 0$ for $|z|>1$. This leads to contradictions in the following two cases.
\begin{itemize}
\item If $\gamma=0$, we have $\operatorname{Re}\{\rho(z)/\sigma(z)\} \equiv 0$ for $|z|>1$. 
The Cauchy--Riemann equations imply that $\rho(z)/\sigma(z)$ is a purely imaginary
constant $\mathrm{i}\beta$. On the other hand, $\rho(x)/\sigma(x)\in\mathbb{R}$ for all
sufficiently large real $x$, and therefore $\beta=0$. This means $\rho(z) \equiv 0$, contradicting that $\rho$ is a nonzero polynomial.

\item If $\gamma>0$, by $\tilde{H}(w)\equiv 0$ for $|w|<1$, we have 
\[
\Delta \tilde H(w)=-4\gamma | ( \tilde{\lambda}(w)/\tilde{\sigma}(w))^\prime |^2\equiv 0\quad \text{for} \quad |w|<1.
\]
Thus $| ( \tilde{\lambda}(w)/\tilde{\sigma}(w))^\prime | \equiv 0$ and $\tilde{\lambda}(w)/\tilde{\sigma}(w)$ is a constant.
It follows that $\operatorname{Re}\{\tilde{h}(w)\}$ is constant in $|w|<1$ since  $ \tilde{H}(w) = \operatorname{Re}\{\tilde{h}(w)\}-\gamma |\tilde{\lambda}(w)/\tilde{\sigma}(w)|^2\equiv 0$. 
Since $\tilde{h}$ is holomorphic, the Cauchy--Riemann equations imply that $\tilde h$ is a constant. Since $s=0$, we have $h(z)=\rho(z)/\sigma(z)$ from \eqref{eq:equi-thm:15}, and therefore 
$\tilde h(w)=\tilde\rho(w)/\tilde\sigma(w)$. Hence there exists a constant $d\in\mathbb C$ such that $\tilde\rho(w)= d \tilde\sigma(w)$,
equivalently,
$\rho(z)=d\,\sigma(z)$.
If $d=0$, then $\rho\equiv 0$ that is a contradiction; if $d\neq 0$, then
$\rho$ and $\sigma$ are proportional, contradicting the relative primality of $\rho$ and $\sigma$. 
\end{itemize}

From the above discussions, $H(z)>0$ for  $|z|>1$. 
\end{enumerate}

\paragraph{Part B: Equivalence of {\rm (ii)} and {\rm (iv)}}

 \textrm{(iv) $\Rightarrow$ (ii):} Assume that (iv) holds. Setting $w = \bar{z}$ in (iv) gives $2\operatorname{Re}\{\rho(z)\sigma(\bar{z})\} - 2\gamma | \lambda(z) |^{2} = |q(z)|^2 + (|z|^2-1)\sum_{j=1}^r |p_j(z)|^2$. For $|z| \geq 1$, $|z|^2-1 \geq 0$, and then the right-hand side is nonnegative, proving (ii).

\textrm{(ii) $\Rightarrow$ (iv):} Assume that (ii) holds. Define the  polynomial 
\begin{equation}
\label{eq:Fzw}
F(z,w) \coloneqq \rho(z)\sigma(w) + \sigma(z)\rho(w) - 2\gamma \lambda(z)\lambda(w).
\end{equation}
By (ii), $F(z, \bar{z}) \ge 0$ on $|z|=1$. According to the Fej\'{e}r--Riesz theorem (see e.g., \cite[Lemma 6.1.3]{daubechies1992ten}), there exists a polynomial $q(z)$ with real coefficients of degree at most $r$ such that 
\begin{equation}
\label{eq:Fz-decomposition}
    F(z,z^{-1}) = q(z)q(z^{-1}),\quad \forall |z|=1.
\end{equation} 
Both $F(z,z^{-1})$ and $q(z)q(z^{-1})$ are Laurent polynomials in $z$.
Since they agree on the unit circle, they agree identically for all $z\neq 0$. 
Define $D(z,w) \coloneqq F(z,w) - q(z)q(w)$. Then $D(z, z^{-1}) \equiv 0$ for $z\neq 0$. Since $D(z,w)$ is of degree $r$ in each variable, we have the expansion $
D(z,w) = \sum_{j=0}^r d_j(z) w^j$, where $d_j(z)$ is the polynomial of degree at most $r$ with real coefficients. Let 
\[
K(z,w) = \sum_{j=0}^{r-1} k_j(z) w^j,
\]
where $k_j(z) = z k_{j-1}(z)-d_j(z)$ for $1\leq j\leq r-1$ and $k_0(z) = -d_0(z)$.
A direct comparison of the coefficients of $w^j$ shows that 
\begin{equation}
    \label{eq:equi-thm:16}
D(z,w)=(zw-1)K(z,w).
\end{equation}
Since $K(z,w)$ is of degree $r-1$ in each variable and $D(z,w)=D(w,z)$, we have $K(z,w) = K(w,z)$. There exists a real symmetric matrix $M\in\mathbb{R}^{r\times r}$ such that
\begin{equation}
\label{eq:equi-thm:2}
K(z,w) = \boldsymbol{z}_{r-1}^\top M \boldsymbol{w}_{r-1},
\end{equation}
where $\boldsymbol{z}_{r-1} = [z^{r-1},z^{r-2},\ldots,1]^\top$ and $\boldsymbol{w}_{r-1} = [w^{r-1},w^{r-2},\ldots,1]^\top$. It remains to show that there exist $r$ linearly independent polynomials $p_1,\ldots,p_r$ with real coefficients and $\deg p_j \leq r-1$ such that
\begin{equation}
\label{eq:equi-thm:0}
K(z,w) =  \sum_{j=1}^r p_j(z)p_j(w),\quad \forall z,w\in\mathbb{C}.
\end{equation}
We proceed by splitting the remaining proof into three steps.

\emph{Step 1: Positive semi-definiteness of $M$.}
Substituting \eqref{eq:equi-thm:2} into \eqref{eq:equi-thm:16}, we obtain
\[
\begin{aligned}
&\boldsymbol{z}_r^\top \boldsymbol{\rho} \boldsymbol{\sigma}^\top \boldsymbol{w}_r
+\boldsymbol{z}_r^\top \boldsymbol{\sigma} \boldsymbol{\rho}^\top \boldsymbol{w}_r
-2\gamma \boldsymbol{z}_r^\top \boldsymbol{\lambda} \boldsymbol{\lambda}^\top \boldsymbol{w}_r - 
\boldsymbol{z}_r^\top \boldsymbol{q} \boldsymbol{q}^\top \boldsymbol{w}_r\\
= &
\boldsymbol{z}_r^\top 
\left(
\begin{bmatrix}
M & \boldsymbol{0}_{r} \\
\boldsymbol{0}_{r}^\top & 0
\end{bmatrix}
-
\begin{bmatrix}
0 & \boldsymbol{0}_{r}^\top \\
\boldsymbol{0}_{r} & M
\end{bmatrix}
\right)
 \boldsymbol{w}_r.
 \end{aligned}
\]
Since this identity holds for any $\boldsymbol{z}_r$ and $\boldsymbol{w}_r$, equating the 
coefficients yields
\[
\boldsymbol{\rho} \boldsymbol{\sigma}^\top 
+ \boldsymbol{\sigma} \boldsymbol{\rho}^\top
-2\gamma  \boldsymbol{\lambda} \boldsymbol{\lambda}^\top 
 - 
\boldsymbol{q} \boldsymbol{q}^\top
=
\begin{bmatrix}
M & \boldsymbol{0}_{r} \\
\boldsymbol{0}_{r}^\top & 0
\end{bmatrix}
-
\begin{bmatrix}
0 & \boldsymbol{0}_{r}^\top \\
\boldsymbol{0}_{r} & M
\end{bmatrix}.
\]
Using the identity $2xy=(x+y)^2-(x-y)^2$, we rewrite this formula as
\begin{equation}
\label{eq:equi-thm:17}
\boldsymbol{a} \boldsymbol{a}^\top 
- \boldsymbol{b} \boldsymbol{b}^\top
-2\gamma  \boldsymbol{\lambda} \boldsymbol{\lambda}^\top 
 - 
\boldsymbol{q} \boldsymbol{q}^\top
=
\begin{bmatrix}
M & \boldsymbol{0}_{r} \\
\boldsymbol{0}_{r}^\top & 0
\end{bmatrix}
-
\begin{bmatrix}
0 & \boldsymbol{0}_{r}^\top \\
\boldsymbol{0}_{r} & M
\end{bmatrix},
\end{equation}
where
$
\boldsymbol{a} = {(\boldsymbol{\rho}+\boldsymbol{\sigma})}/{\sqrt{2}}$
and
$
\boldsymbol{b} = {(\boldsymbol{\rho}-\boldsymbol{\sigma})}/{\sqrt{2}}.
$
We write
\[
\boldsymbol{a} = [a_r,a_{r-1},\ldots,a_0]^\top
\quad \text{and}\quad 
\boldsymbol{b} = [b_r,b_{r-1},\ldots,b_0]^\top.
\]
Then $a_i = (\rho_i+\sigma_i)/\sqrt{2}$ and $b_i = (\rho_i-\sigma_i)/\sqrt{2}$ for $i=0,\ldots,r$.

We first show that $a_r\neq 0$. Indeed,
if $a_r=0$, then $\sigma_r=-\rho_r\neq 0$. Hence, as $x \to \infty$ on the real line, the leading term of $F(x,x)$ is $2\rho_r\sigma_r x^{2r} - 2\gamma \lambda_r^2 x^{2r} = -2(\rho_r^2+\gamma \lambda_r^2 )x^{2r} < 0$, which contradicts $F(x,x) \ge 0$ for $x \ge 1$.

Since $a_r\neq0$, the columns of
\begin{equation}
\label{eq:equi-thm:24}
L \coloneqq
\begin{bmatrix}
-\dfrac{a_{r-1}}{a_r} & -\dfrac{a_{r-2}}{a_r} & \cdots & -\dfrac{a_0}{a_r} \\
1 & 0 & \cdots & 0 \\
0 & 1 & \cdots & 0 \\
\vdots & \vdots & \ddots & \vdots \\
0 & 0 & \cdots & 1
\end{bmatrix}
\in \mathbb{R}^{(r+1)\times r},
\end{equation}
form a basis of null space of the row vector $\boldsymbol{a}^\top$. Then 
$\boldsymbol{a}^\top L=\boldsymbol{0}_r^\top$. 

Multiplying \eqref{eq:equi-thm:17} from the left by $L^\top$ and right by $L$ yields 
\begin{equation}
\label{eq:equi-thm:18}
    M - J^\top M J = Q,
\end{equation}
where $ Q\coloneqq L^\top(\boldsymbol{b}\boldsymbol{b}^\top + 2\gamma \boldsymbol{\lambda}\boldsymbol{\lambda}^\top + \boldsymbol{q}\boldsymbol{q}^\top)L$ 
and
\[
J\coloneqq
\begin{bmatrix}
-\dfrac{a_{r-1}}{a_r} & -\dfrac{a_{r-2}}{a_r} & \cdots & -\dfrac{a_1}{a_r} & -\dfrac{a_0}{a_r} \\
1 & 0 & \cdots & 0 & 0 \\
0 & 1 & \cdots & 0 & 0 \\
\vdots & \vdots & \ddots & \vdots & \vdots \\
0 & 0 & \cdots & 1 & 0
\end{bmatrix}
\in \mathbb{R}^{r\times r}.
\]

If the spectral radius of $J$ is strictly less than $1$, then \eqref{eq:equi-thm:18} admits a solution in the form of convergent matrix series \cite[Eqs.~(2)--(3)]{smith1968}: 
\begin{equation}
\label{eq:equi-thm:19}
M=\sum_{m=0}^{\infty}(J^\top)^m QJ^m .
\end{equation}
Since $Q$ is positive semi-definite, \eqref{eq:equi-thm:19} implies that $M$ is also positive semi-definite.

We now show that the spectral radius of $J$ is strictly less than $1$. Define
\[
A(z) = \boldsymbol{z}_r^\top \boldsymbol{a},\quad 
B(z) = \boldsymbol{z}_r^\top \boldsymbol{b}.
\]
Then 
\[
A(z) = \frac{\rho(z)+\sigma(z)}{\sqrt{2}},\quad B(z) = \frac{\rho(z)-\sigma(z)}{\sqrt{2}},
\]
and therefore,
\[
F(z,w) = A(z)A(w) - B(z)B(w) - 2\gamma \lambda(z)\lambda(w). 
\]
By condition (ii), we have $F(z,\bar{z})\geq 0$ for $|z|\geq 1$, or, equivalently,
\[
|A(z)|^2 \geq |B(z)|^2 + 2\gamma |\lambda(z)|^2, \quad \text{for}\quad |z|\geq 1.
\]
Note that the eigenvalues of $J$ are precisely the roots of $A(z)$, counted with algebraic multiplicity. It remains to show that all zeros of $A(z)$ lie in the open unit disk.
Suppose by contradiction that $A(z_0)=0$ for some $|z_0|\geq 1$. Then 
\[
|B(z_0)|^2 + 2\gamma |\lambda(z_0)|^2 \leq |A(z_0)|^2 = 0.
\]
Hence $B(z_0)=0$. This implies $\rho(z_0)=\sigma(z_0)=0$, contradicting the relative primality of $\rho$ and $\sigma$. Thus, all roots of $A(z)$ lie in the open unit disk. Hence, the spectral radius of $J$ is strictly less than $1$, and the positive
semi-definiteness of $M$ follows from \eqref{eq:equi-thm:19}.

\emph{Step 2: Positive definiteness of $M$.}
We proceed by contradiction and assume that $M$ is not strictly positive definite. Then there must exist a nonzero real vector $\boldsymbol{x} \in \mathbb{R}^r$ such that $\boldsymbol{x}^\top M \boldsymbol{x} = 0$, which implies  
\[
\sum_{m=0}^\infty \boldsymbol{x}^\top (J^\top)^m Q J^m \boldsymbol{x} = 0.
\]
Since $Q$ is positive semi-definite and  each term in the sum is a nonnegative real number, the infinite sum vanishes if and only if every individual term vanishes.
Consequently, $\boldsymbol{x}^\top (J^\top)^m Q J^m \boldsymbol{x} = 0$ for $m\geq 0$. By the definition of $Q$, in particular, we have $\boldsymbol{b}^\top L J^m \boldsymbol{x} = 0$ for $m\geq 0$. Set $\boldsymbol{c}\coloneqq  L^\top \boldsymbol{b} \in \mathbb{R}^r$. Then
\[
\boldsymbol{c}^\top J^m \boldsymbol{x} = 0\quad \text{for}\quad m=0,\ldots,r-1.
\]
Thus, the following linear system has a nonzero solution $\boldsymbol{x}$:
\[
\begin{bmatrix} \boldsymbol{c}^\top  \\ \boldsymbol{c}^\top  J \\ \vdots \\ \boldsymbol{c}^\top  J^{r-1} \end{bmatrix} \boldsymbol{x} = \boldsymbol{0}_{r}.
\]
Equivalently, the column vectors $\boldsymbol{c}, J^\top \boldsymbol{c},\ldots, (J^\top)^{r-1} \boldsymbol{c}$ are linearly dependent, and therefore the Krylov subspace $\mathcal{K}_r (J^\top, \boldsymbol{c})\coloneqq \operatorname{span}\{\boldsymbol{c}, J^\top \boldsymbol{c},\ldots, (J^\top)^{r-1} \boldsymbol{c}\}$ has dimension less than $r$. Since $A(z)/a_r$ is the characteristic polynomial of $J^\top$, the Cayley--Hamilton theorem gives
\[
A(J^\top) = a_r (J^\top)^r + a_{r-1} (J^\top)^{r-1} + \ldots + a_0 I_r = 0,
\]
where $I_r\in\mathbb{R}^{r\times r}$ is the identity matrix. 
Since $a_r\neq 0$, it follows that $(J^\top)^r \boldsymbol{c} \in \mathcal{K}_r (J^\top, \boldsymbol{c})$. Therefore  $\mathcal{K}_r (J^\top, \boldsymbol{c})$ is an invariant subspace of $J^\top$, since $\boldsymbol{y}\in \mathcal{K}_r (J^\top, \boldsymbol{c})$ implies $J^\top \boldsymbol{y}\in \mathcal{K}_r (J^\top, \boldsymbol{c})$.

Let $\mathcal{V}\coloneqq \mathcal{K}_r (J^\top, \boldsymbol{c})^\perp$ 
be the orthogonal complement space of $\mathcal{K}_r (J^\top, \boldsymbol{c})$. Then $\mathcal{V}$ is non-trivial. Moreover, for any $\boldsymbol{v}\in \mathcal{V}$, we have 
\begin{equation}
\label{eq:equi-thm:21}
\boldsymbol{v}^\top \boldsymbol{y} = 0, \quad \forall \boldsymbol{y}\in \mathcal{K}_r (J^\top, \boldsymbol{c}).
\end{equation}
This implies
$
(J \boldsymbol{v})^\top \boldsymbol{y} = \boldsymbol{v}^\top (J^\top \boldsymbol{y}) = 0$,
since $J^\top \boldsymbol{y}\in \mathcal{K}_r (J^\top, \boldsymbol{c})$. Then $J \boldsymbol{v}\in\mathcal{V}$, and therefore $\mathcal{V}$ is an invariant subspace of $J$. Suppose that the dimension of $\mathcal{V}$ is $s$. Then there exists $S\in\mathbb{R}^{r\times s}$, whose column vectors form a basis of $\mathcal{V}$. By $J \mathcal{V}\subset \mathcal{V}$, there exists a nonzero matrix $T\in\mathbb{R}^{s\times s}$ such that 
$
JS = ST$.
Then there exist $\lambda_*\in\mathbb{C}$ and $\boldsymbol{0}_{s}\neq \tilde{\boldsymbol{v}}_* \in \mathbb{C}^s$ such that
$
T \tilde{\boldsymbol{v}}_* = \lambda_* \tilde{\boldsymbol{v}}_*.
$
Let $\boldsymbol{v}_* = S\tilde{\boldsymbol{v}}_*$. Then $\boldsymbol{0}_{r} \neq \boldsymbol{v}_* \in \mathbb{C}^r$ and
\begin{equation}
\label{eq:equi-thm:20}
J \boldsymbol{v}_* = J S \tilde{\boldsymbol{v}}_*
= ST \tilde{\boldsymbol{v}}_* = \lambda_* S \tilde{\boldsymbol{v}}_* =   \lambda_* \boldsymbol{v}_*.
\end{equation}
Moreover,
$
\boldsymbol{v}_*^\top \boldsymbol{c} = 
\tilde{\boldsymbol{v}}_*^\top S^\top \boldsymbol{c} = 0.
$

Since $\lambda_*$ is an eigenvalue of $J$, we have $A(\lambda_*)=0$. 
Write $\boldsymbol{v}_*=[v_{r-1},\ldots,v_0]^\top$. Substituting this into  \eqref{eq:equi-thm:20}, we obtain
$
v_{j+1}=\lambda_* v_j$ for $j=0,\ldots,r-2.
$
Thus
\begin{equation}
\label{eq:equi-thm:22}
\boldsymbol{v}_* = v_0 [\lambda_*^{r-1},\ldots,\lambda_*,1]^\top.
\end{equation}
Moreover, $v_0\neq0$; otherwise $\boldsymbol{v}_*=\boldsymbol{0}_{r}$.
Without loss of generality, we can scale the eigenvector $\boldsymbol{v}_*$ such that $v_0=1$. Since $A(\lambda_*)=0$ and $a_r\neq 0$, we have
\begin{equation}
\label{eq:equi-thm:23}
- \frac{1}{a_r} \sum_{i=0}^{r-1} a_i \lambda_{*}^i = \lambda_{*}^{r}.
\end{equation}
Using the definition of $L$ in \eqref{eq:equi-thm:24}, and combining \eqref{eq:equi-thm:22} with \eqref{eq:equi-thm:23}, we have
\[
L \boldsymbol{v}_* = [\lambda_*^{r},\lambda^{r-1}_*,\ldots,\lambda_*,1]^\top.
\]

Since $\boldsymbol{v}_*^\top\boldsymbol{c}=0$, using $\boldsymbol{c}=L^\top \boldsymbol{b}$, we have
\[
0 =   \boldsymbol{v}_*^\top \boldsymbol{c} = 
 \boldsymbol{v}_*^\top L^\top \boldsymbol{b}
 = (L \boldsymbol{v}_*)^\top \boldsymbol{b}
 = [\lambda_*^{r},\lambda^{r-1}_*,\ldots,\lambda_*,1] \boldsymbol{b} = B(\lambda_*).
\]
Then the complex number $\lambda_*$ is a common root of both polynomials $A(z)$ and $B(z)$. We have
$
\rho(\lambda_*) = [A(\lambda_*) + B(\lambda_*)]/{\sqrt{2}} = 0$ and $\sigma(\lambda_*) = [A(\lambda_{*}) - B(\lambda_{*})]/{\sqrt{2}} = 0$.
That is, $\rho(z)$ and $\sigma(z)$ share a common root $z = \lambda_{*} \in \mathbb{C}$. This  contradicts the assumption that the polynomials $\rho(z)$ and $\sigma(z)$ are relatively prime.

\emph{Step 3: Derivation of \eqref{eq:equi-thm:0}.} 

Since $M$ is real positive definite, let $M=PP^\top$ be its Cholesky
factorization, where $P=[\boldsymbol p_1,\ldots,\boldsymbol p_r]\in\mathbb R^{r\times r}$.
Define $p_j(z)=\boldsymbol z_{r-1}^\top \boldsymbol p_j$.
Then each $p_j$ has real coefficients and degree at most $r-1$, and
\[
K(z,w)=\boldsymbol z_{r-1}^\top M \boldsymbol w_{r-1}
      =\sum_{j=1}^r p_j(z)p_j(w).
\]
Because $P$ is nonsingular, the polynomials $p_1,\ldots,p_r$ are linearly independent.

\paragraph{Part C: Equivalence of {\rm (iv)} and {\rm (v)}} 

Condition \textup{(iv)} is equivalent to the identity
\begin{equation}
\label{eq:equi-thm:25}
\boldsymbol{z}^\top_r (\boldsymbol{\rho} \boldsymbol{\sigma}^\top + 
\boldsymbol{\sigma} \boldsymbol{\rho}^\top  -2\gamma \boldsymbol{\lambda} \boldsymbol{\lambda}^\top )
\boldsymbol{w}_r
=
\boldsymbol{z}_r^\top \boldsymbol{q}\boldsymbol{q}^\top \boldsymbol{w}_r
+
\boldsymbol{z}_r^\top 
\left(
\begin{bmatrix}
G & \boldsymbol{0}_{r} \\
\boldsymbol{0}_{r}^\top & 0
\end{bmatrix}
-
\begin{bmatrix}
0 & \boldsymbol{0}_{r}^\top \\
\boldsymbol{0}_{r} & G
\end{bmatrix}
\right)
 \boldsymbol{w}_r
\end{equation}
for all $\boldsymbol{z}_r$ and $\boldsymbol{w}_r$, 
where $G=PP^\top$ and $P=[\boldsymbol{p}_1,\ldots,\boldsymbol{p}_r]\in\mathbb{R}^{r\times r}$ with $\boldsymbol{p}_j\in\mathbb{R}^r$ being the corresponding coefficient vector of $p_j$. Equating the coefficients in \eqref{eq:equi-thm:25} gives
\eqref{eq:equi-thm:cond5}. Moreover, the polynomials
$p_1,\ldots,p_r$ are linearly independent if and only if their coefficient
vectors $\boldsymbol{p}_1,\ldots,\boldsymbol{p}_r$ are linearly
independent. This is equivalent to the nonsingularity of $P$, and hence to
the positive definiteness of $G=PP^\top$.
\end{proof}

% Let $\boldsymbol{p}_j\in\mathbb{R}^r$ be the corresponding coefficient vector of $p_j$, and set
% $P=[\boldsymbol{p}_1,\ldots,\boldsymbol{p}_r]\in\mathbb{R}^{r\times r}$ and $G=PP^\top$.
% Then
% $
% \sum_{j=1}^r p_j(z)p_j(w)
% =
% \boldsymbol{z}_{r-1}^\top G\,\boldsymbol{w}_{r-1}$.
% Therefore, condition \textup{(iv)} is equivalent to
% \begin{equation}
% \label{eq:equi-thm:11}
% \boldsymbol{z}^\top_r (\boldsymbol{\rho} \boldsymbol{\sigma}^\top + 
% \boldsymbol{\sigma} \boldsymbol{\rho}^\top  -2\gamma \boldsymbol{\lambda} \boldsymbol{\lambda}^\top )
% \boldsymbol{w}_r
% =
% \boldsymbol{z}_r^\top \boldsymbol{q}\boldsymbol{q}^\top \boldsymbol{w}_r
% +
% (zw-1)\boldsymbol{z}_{r-1}^\top G\,\boldsymbol{w}_{r-1},
% \end{equation}
% Note that
% \begin{equation}
% \label{eq:equi-thm:12}
% \boldsymbol{z}_r^\top
% \begin{bmatrix}
% G & \\
% & 0
% \end{bmatrix}
% \boldsymbol{w}_r
% =zw\,\boldsymbol{z}_{r-1}^\top G\,\boldsymbol{w}_{r-1}
% \quad\text{and}\quad
% \boldsymbol{z}_r^\top
% \begin{bmatrix}
% 0 & \\
% & G
% \end{bmatrix}
% \boldsymbol{w}_r
% =\boldsymbol{z}_{r-1}^\top G\,\boldsymbol{w}_{r-1}.
% \end{equation}
% Substituting \eqref{eq:equi-thm:12} into \eqref{eq:equi-thm:11} and equating the coefficients gives \eqref{eq:equi-thm:cond5}.  Moreover, the polynomials $p_1,\ldots,p_r$ are linearly independent if and only if the vectors $\boldsymbol{p}_1,\ldots,\boldsymbol{p}_r$ are linearly independent, i.e., if and only if $P$ is nonsingular, which is equivalent to $G=PP^\top$ being positive definite.

\cref{thm:equivalence} provides five equivalent statements involving the polynomials $\rho(z)$, $\sigma(z)$, and $\lambda(z)$, together with a constant $\gamma$. These statements include the corresponding inequalities outside the unit disk and on its closed exterior in conditions (i) and (ii), the corresponding inequality on the unit circle and root conditions in condition (iii), the polynomial decomposition in condition (iv), and the matrix decomposition associated with the polynomial coefficient vectors in condition (v). In particular, when $\gamma=0$ and $\lambda(z)=z^r$, the equivalence between conditions (ii) and (iv) recovers the result of \cite{BaiocchiCrouzeix1989}, while the implication from condition (i) to condition (iv) recovers \cite[Lemma 3.1]{Dahlquist-1978-G-stability}.

It should be noted that condition (iii) reduces the verification of $\operatorname{Re}\{\rho(z)\sigma(\bar z)\}\ge \gamma|\lambda(z)|^2$ from the closed exterior of the unit disk to the unit circle, supplemented by the root conditions for $\sigma(z)$. This characterization motivates the construction of the feasibility problem \eqref{FP} in \Cref{sec:feasibility}. Moreover, the matrix $G$ whose existence is asserted in condition (v) is essential for constructing the positive definite matrices in \cref{lemma:multiplier-LMM-energy}. 
Thus, the equivalence between conditions (iii) and (v) reduces the construction of these positive definite matrices to the verification of algebraic conditions on the associated polynomials, which is presented in \Cref{thm:main}.

\begin{remark}
\label{rem:determine-G}
 Once $\boldsymbol{\rho}$, $\boldsymbol{\sigma}$, $\boldsymbol{\lambda}$, and
$\gamma$ in \eqref{eq:equi-thm:cond5} are fixed, the vector $\boldsymbol{q}$ and the matrix $G$
% in \eqref{eq:equi-thm:cond5}
can be determined explicitly. First,
$\boldsymbol{q}$ is the coefficient vector of the polynomial $q(z)$ in
\eqref{eq:equi-thm:cond4}. It is obtained from 
\eqref{eq:Fz-decomposition} by the Fej\'er--Riesz theorem; see 
\cite[Step~2 in Section~3.3]{QuanWangWangXuPre} for the detailed derivation of $\boldsymbol{q}$. Then $G$ can be obtained from \eqref{eq:equi-thm:cond5}.

% We next compute $G$. Multiplying \eqref{eq:equi-thm:cond5} from the left by $P_0$ and from
% the right by $P_0^\top$, where $P_0=[I_r, \boldsymbol{0}] \in \mathbb{R}^{r\times (r+1)}$, gives
% \[
% G-SGS^\top = P_0 \left( \boldsymbol{\rho} \boldsymbol{\sigma}^\top + \boldsymbol{\sigma} \boldsymbol{\rho}^\top - 2\gamma \boldsymbol{\lambda} \boldsymbol{\lambda}^\top - \boldsymbol{q} \boldsymbol{q}^\top \right ) P_0^\top,
% \]
% where
% \[
% S=\begin{bmatrix}
% 0 & 0 & \cdots & 0 & 0 \\
% 1 & 0 & \cdots & 0 & 0 \\
% 0 & 1 & \cdots & 0 & 0 \\
% \vdots & \vdots & \ddots & \vdots & \vdots \\
% 0 & 0 & \cdots & 1 & 0
% \end{bmatrix}
% \in \mathbb{R}^{r\times r}.
% \]
% Then
% \[
% G = \sum_{m=0}^{r-1} S^m  P_0 \left( \boldsymbol{\rho} \boldsymbol{\sigma}^\top + \boldsymbol{\sigma} \boldsymbol{\rho}^\top - 2\gamma \boldsymbol{\lambda} \boldsymbol{\lambda}^\top - \boldsymbol{q} \boldsymbol{q}^\top \right ) P_0^\top (S^\top)^m.
% \]
\end{remark}

\subsection{Dissipation, nonnegativity, and consistency of the modified energy}
\label{subsec:properties-mod-energy}

In this subsection, we establish several properties of the modified energy $E_G^n$ defined in \eqref{eq:EG}. Its dissipation property is derived under the positive-definiteness conditions given in \cref{lemma:multiplier-LMM-energy}, while sufficient conditions for its non-negativity are given in \cref{lem:1}. The consistency result in \cref{rem:consistency} follows from a property of the matrix $G_B$ stated in \cref{lem:1gnu1}. These results are then  combined in \cref{thm:main} to prove the modified energy dissipation of IMEX-LMMs.

% In this subsection, we establish three properties of the modified energy $E_G^n$, namely dissipation, non-negativity, and consistency. First, \cref{lemma:multiplier-LMM-energy} gives sufficient positive-definiteness conditions for the modified energy dissipation law of IMEX-LMMs. Next, \cref{lem:1} proves the uniform non-negativity of $E_G^n$.
% {\blue Finally, \cref{lem:1gnu1} establishes a property of the matrix $G_B$, from which \cref{rem:consistency} deduces the consistency of $E_G^n$ with the original energy.}
% Henceforth, $\boldsymbol{0}_{m\times n}$ denotes the zero matrix in $\mathbb{R}^{m\times n}$; whenever the dimension is clear from the context, the subscripts are omitted.

\begin{lemma}[Positive-definiteness conditions for energy dissipation]
\label{lemma:multiplier-LMM-energy}
Assume that the gradient flow \eqref{model:phase-field-eq} has a Lipschitz continuous nonlinearity satisfying \eqref{eq:lip-f}, and that there exist constants $\zeta>0$ and $0<\eta\le 1$ such that \eqref{eq:l2-control} holds. Consider the $k$-step IMEX-LMM \eqref{scheme:lmm-diff-form}, and let $\boldsymbol{a}^{(k)}$, $\boldsymbol{B}^{(k)}$, and $\hat{\boldsymbol{b}}^{(k)}$ be the coefficient vectors defined in \eqref{eq:coef-vecs}. Let the multiplier in \eqref{eq:multiplier1} be specified by the vectors $\boldsymbol{\mu}$ and $\boldsymbol{\nu}$ defined in \eqref{eq:multiplier2}. 
Suppose that there exist two positive definite matrices $G_a\in \mathbb{R}^{(k-1)\times (k-1)}$ and $G_B\in \mathbb{R}^{k\times k}$, and two constants $\alpha, \beta> 0$ satisfying the following positive-definiteness conditions:
\begin{equation}
\label{eq:pdc}
\boldsymbol{x}^\top U_a \boldsymbol{x} \geq \alpha x_1^2,\quad
\boldsymbol{y}^\top U_B \boldsymbol{y} \geq \beta (y_1-y_2)^2,
\end{equation}
for any $\boldsymbol{x}=[x_1,\ldots,x_k]^\top\in\mathbb{R}^k$ and $\boldsymbol{y} = [y_1,y_2,\ldots,y_{k+1}]^\top\in \mathbb{R}^{k+1}$,
where 
\begin{align}
U_a &= \frac{1}{2}\bigl(\boldsymbol{\mu} (\boldsymbol{a}^{(k)})^\top + \boldsymbol{a}^{(k)} \boldsymbol{\mu}^\top\bigr) - 
\begin{bmatrix} G_a & \boldsymbol{0}_{k-1} \\ \boldsymbol{0}_{k-1}^\top & 0 \end{bmatrix} + \begin{bmatrix} 0 & \boldsymbol{0}_{k-1}^\top \\ \boldsymbol{0}_{k-1} & G_a \end{bmatrix}, \label{eq:Ua} \\[1ex]
U_B &= \frac{1}{2}\bigl(\boldsymbol{\nu} (\boldsymbol{B}^{(k)})^\top + \boldsymbol{B}^{(k)} \boldsymbol{\nu}^\top\bigr) - 
\begin{bmatrix} G_B & \boldsymbol{0}_{k} \\ \boldsymbol{0}_{k}^\top & 0 \end{bmatrix} + \begin{bmatrix} 0 & \boldsymbol{0}_{k}^\top \\ \boldsymbol{0}_{k} & G_B \end{bmatrix}. \label{eq:UB}
\end{align}  
Then the modified energy $E_G^n$ defined in \eqref{eq:EG} satisfies $E_G^{n+1} \leq E_G^n$ under the time-step restriction:
\begin{equation}
\label{eq:tau-max}
0<\tau \leq \tau_{\max}, \quad \tau_{\max} \coloneqq 
\frac{\alpha \beta^{\bar{\eta}}}{(\ell_f \hat{c}_0/2) ^{1+\bar{\eta}} \eta (1-\eta)^{\bar{\eta}} \zeta^{2+2\bar{\eta}} },
\end{equation}
where $\ell_f$ is the Lipschitz constant in \eqref{eq:lip-f}, $\hat{c}_0 =\sum_{j=0}^{k-1} \tilde{c}_j$ with $\tilde{c}_j$ defined in \eqref{eq:tilde-c}, and $\bar{\eta}\coloneqq \frac{1-\eta}{\eta}$. Here and throughout, we adopt the convention $0^0=1$.
\end{lemma}

\begin{proof}
Taking the inner product of \eqref{scheme:lmm-diff-form} with $\sum_{i=0}^{k-1} \mu_i \delta u^{n+1-i}$, we obtain
\begin{equation}
\label{eq:thm1:1}
I_1 + I_2  = I_3 + I_4 ,
\end{equation}
where
\[
\begin{aligned}
 & I_1 = \frac{1}{\tau} \biggl( \sum_{i=0}^{k-1} a_i^{(k)} \mathcal{M}^{-1} \delta u^{n+1-i},  \sum_{i=0}^{k-1} \mu_i \delta u^{n+1-i}  \biggr), 
   \\
 & I_2 = -\biggl( \sum_{i=0}^{k} B_i^{(k)} \mathcal{L}  u^{n+1-i},  \sum_{i=0}^{k-1} \mu_i \delta u^{n+1-i} \biggr ) = 
 -\biggl( \sum_{i=0}^{k} B_i^{(k)} \mathcal{L}  u^{n+1-i},  \sum_{i=0}^{k} \nu_i  u^{n+1-i} \biggr ) , \\
& I_3 = \biggl( f(u^n), \sum_{i=0}^{k-1} \mu_i \delta u^{n+1-i}  \biggr ), \;
 I_4=\biggl( \sum_{i=1}^{k-1} \hat{b}_i^{(k)} \delta f(u^{n+1-i}),   \sum_{i=0}^{k-1} \mu_i \delta u^{n+1-i}\biggr ).
 \end{aligned}
\]
Let $\boldsymbol{w}_n = [\delta u^n,\delta u^{n-1},\ldots,\delta u^{n+1-k}]^\top$ and $\tilde{\boldsymbol{w}}_n = [ u^n, u^{n-1},\ldots, u^{n-k}]^\top$. 
Using \eqref{eq:UB}, we have
\begin{equation}
\label{eq:thm1:234-1}
\begin{aligned}
I_2 & = - \frac{1}{2} \left[ (\tilde{\boldsymbol{w}}_{n+1}, \mathcal{L} \tilde{\boldsymbol{w}}_{n+1})_{\boldsymbol{\nu} (\boldsymbol{B}^{(k)})^\top}
+
(\tilde{\boldsymbol{w}}_{n+1}, \mathcal{L} \tilde{\boldsymbol{w}}_{n+1})_{\boldsymbol{B}^{(k)} \boldsymbol{\nu}^\top}
\right ] \\
& = - (\tilde{\boldsymbol{w}}_{n+1},\mathcal{L} \tilde{\boldsymbol{w}}_{n+1})_{U_B}  - (\boldsymbol{u}_{n+1}, \mathcal{L} \boldsymbol{u}_{n+1})_{G_B} + (\boldsymbol{u}_{n},  \mathcal{L} \boldsymbol{u}_{n})_{G_B},
\end{aligned}
\end{equation}

% By multiplying \eqref{eq:UB} from the right by $\mathcal{L} \tilde{\boldsymbol{w}}_{n+1}$, and taking the inner product with $ \tilde{\boldsymbol{w}}_{n+1}$, we obtain
% \begin{equation}
% \label{eq:thm1:234-1}
% I_2
% =  - (\tilde{\boldsymbol{w}}_{n+1},\mathcal{L} \tilde{\boldsymbol{w}}_{n+1})_{U_B}  - (\boldsymbol{u}_{n+1}, \mathcal{L} \boldsymbol{u}_{n+1})_{G_B} + (\boldsymbol{u}_{n},  \mathcal{L} \boldsymbol{u}_{n})_{G_B},
% \end{equation}
where $\boldsymbol{u}_n$ is defined in \eqref{eq:eng-vn-un}. 
Analogously, $I_1$ can be rewritten using \eqref{eq:Ua} as
\begin{equation}
\label{eq:thm1:234-2}
I_1 
= \frac{1}{\tau} ( \boldsymbol{w}_{n+1},  \mathcal{M}^{-1} \boldsymbol{w}_{n+1})_{U_a} + \frac{1}{\tau}  ( \boldsymbol{v}_{n+1}, \mathcal{M}^{-1} \boldsymbol{v}_{n+1})_{G_a}
- \frac{1}{\tau} (\boldsymbol{v}_{n}, \mathcal{M}^{-1} \boldsymbol{v}_{n})_{G_a}, 
\end{equation}
where $\boldsymbol{v}_n$ is defined in \eqref{eq:eng-vn-un}. 

We divide $I_3$ into two parts:
\begin{equation}
\label{eq:thm1:5}
I_3 = \biggl( f(u^n), \sum_{i=0}^{k-1} \mu_i \delta u^{n+1-i}  \biggr ) = \sum_{i=0}^{k-1} \mu_i (f(u^n), u^{n+1-i}-u^{n-i}) = J_1 + J_2,
\end{equation}
where
\[
J_1 = \sum_{i=1}^{k-1} \mu_i (f(u^n)-f(u^{n-i}), u^{n+1-i} - u^{n-i}), \ J_2 = \sum_{i=0}^{k-1} \mu_i (f(u^{n-i}), u^{n+1-i} - u^{n-i}).
\]
By \eqref{eq:lip-f}, the Cauchy--Schwarz inequality, and $2ab\leq a^2+b^2$, we have
\begin{equation}
\label{eq:thm1:7}
\begin{aligned}
J_1  &  \geq -\sum_{i=1}^{k-1} |\mu_i| \|f(u^{n})-f(u^{n-i})\| \|\delta u^{n+1-i}\| 
 \geq - \ell_f \sum_{i=1}^{k-1} |\mu_i| \| u^n-u^{n-i} \|  \|\delta u^{n+1-i}\| \\
& \geq - \frac{\ell_f}{2} \sum_{i=1}^{k-1} |\mu_i| \bigl ( \| u^n-u^{n-i} \|^2+  \|\delta u^{n+1-i}\|^2 \bigr ) \\
& \geq - \frac{\ell_f}{2} \sum_{i=1}^{k-1} |\mu_i| \biggl ( i \sum_{j=1}^i \left \|  \delta u^{n+1-j} \right \|^2+  \|\delta u^{n+1-i}\|^2 \biggr ) \\
&  = - \frac{\ell_f}{2} \sum_{i=1}^{k-1} \Bigl(|\mu_i| + \sum_{j=i}^{k-1} j |\mu_j| \Bigr) \|\delta u^{n+1-i}\|^2,
\end{aligned}
\end{equation}
where in the last inequality, we used the identity $u^n-u^{n-i} = \sum_{j=1}^i \delta u^{n+1-j}$ and the inequality $\|\sum_{j=1}^i x_j\|^2\leq i \sum_{j=1}^i\|x_j\|^2$. 
To estimate $J_2$, we apply \eqref{eq:lip-F} with $u=u^{n+1-i}$ and $v=u^{n-i}$ to obtain
\begin{equation}
\label{eq:thm1:8}
|(F(u^{n+1-i}) - F(u^{n-i}), 1 ) - (f(u^{n-i}), u^{n+1-i}-u^{n-i})|   \leq \frac{\ell_f}{2} \| u^{n+1-i}-u^{n-i}\|^2.  
\end{equation}
Then
\begin{equation}
\label{eq:thm1:9}
\mu_i (f(u^{n-i}), u^{n+1-i} - u^{n-i}) \geq \mu_i (F(u^{n+1-i})-F(u^{n-i}),1) - \frac{\ell_f}{2} |\mu_i| \| u^{n+1-i} - u^{n-i}\|^2.
\end{equation}
Summing \eqref{eq:thm1:9} over $i$ gives
\begin{equation}
\label{eq:thm1:10}
J_2 \geq \sum_{i=0}^{k-1}  \mu_i (F(u^{n+1-i})-F(u^{n-i}),1) - \frac{\ell_f}{2}  \sum_{i=0}^{k-1} |\mu_i| \| u^{n+1-i} - u^{n-i}\|^2.
\end{equation}
Substituting \eqref{eq:thm1:7} and \eqref{eq:thm1:10} into \eqref{eq:thm1:5} gives
\begin{equation}
\label{eq:thm1:11}
\begin{aligned}
I_3 \geq  & - \frac{\ell_f}{2} \sum_{i=1}^{k-1} \biggl(|\mu_i| + \sum_{j=i}^{k-1} j |\mu_j| \biggr) \|\delta u^{n+1-i}\|^2  \\
& + \sum_{i=0}^{k-1}  \mu_i (F(u^{n+1-i})-F(u^{n-i}),1) - \frac{\ell_f}{2}  \sum_{i=0}^{k-1} |\mu_i| \| u^{n+1-i} - u^{n-i}\|^2.
\end{aligned}
\end{equation}

For $I_4$, by \eqref{eq:lip-f}, the Cauchy--Schwarz inequality and $2ab\le a^2+b^2$, we obtain
\begin{equation}
\label{eq:thm1:12}
% I_4\ge -\frac{\ell_f}{2}\sum_{i=0}^{k-1}\biggl((1-\delta_{i,0})|\hat b_i^{(k)}|\sum_{j=0}^{k-1}|\mu_j|
% +|\mu_i|\sum_{j=1}^{k-1}|\hat b_j^{(k)}|\biggr)\|\delta u^{n+1-i}\|^2.
\begin{aligned}
I_4  &
 \geq -\ell_f \sum_{i=1}^{k-1}  \sum_{j=0}^{k-1} | \hat{b}_i^{(k)}| |\mu_j|  \| \delta u^{n+1-i} \| \|\delta u^{n+1-j}\| \\
& \geq - \frac{\ell_f}{2} \sum_{i=1}^{k-1}  \sum_{j=0}^{k-1} | \hat{b}_i^{(k)}|  |\mu_j| \left ( \| \delta u^{n+1-i} \|^2 +  \|\delta u^{n+1-j}\|^2 \right ) \\
% & = - \frac{\ell_f}{2} \sum_{i=1}^{k-1}  | \hat{b}_i^{(k)}| \left( \sum_{j=0}^{k-1}   |\mu_j| \right )   \| \delta u^{n+1-i} \|^2 - \frac{\ell_f}{2} \sum_{j=0}^{k-1} |\mu_j| \left( \sum_{i=1}^{k-1} | \hat{b}_i^{(k)}| \right) \|\delta u^{n+1-j}\|^2 \\
& = -\frac{\ell_f}{2} \sum_{i=0}^{k-1}  \left (  (1- \delta_{i,0}) | \hat{b}_i^{(k)}|  \sum_{j=0}^{k-1}   |\mu_j|   +  |\mu_i| \sum_{j=1}^{k-1} | \hat{b}_j^{(k)}|  \right )  \|\delta u^{n+1-i}\|^2.
\end{aligned}
\end{equation}

Substituting \eqref{eq:thm1:234-1}, \eqref{eq:thm1:234-2}, \eqref{eq:thm1:11} and \eqref{eq:thm1:12} into \eqref{eq:thm1:1}, we have
\begin{equation}
\label{eq:thm1:13}
\begin{aligned}
& -\frac{1}{\tau} ( \boldsymbol{v}_{n+1}, \mathcal{M}^{-1} \boldsymbol{v}_{n+1})_{G_a}
 + ( \boldsymbol{u}_{n+1}, \mathcal{L} \boldsymbol{u}_{n+1})_{G_B}   \\
& + \frac{1}{\tau} ( \boldsymbol{v}_{n}, \mathcal{M}^{-1} \boldsymbol{v}_{n})_{G_a}
- ( \boldsymbol{u}_{n},  \mathcal{L} \boldsymbol{u}_{n})_{G_B}
 +\sum_{i=0}^{k-1} \mu_i (F(u^{n+1-i}) - F(u^{n-i}), 1 )  \\
& \leq \frac{1}{\tau}  (\boldsymbol{w}_{n+1},  \mathcal{M}^{-1} \boldsymbol{w}_{n+1})_{U_a} - (\tilde{\boldsymbol{w}}_{n+1}, \mathcal{L} \tilde{\boldsymbol{w}}_{n+1})_{U_B} 
 + \frac{\ell_f}{2} \sum_{i=0}^{k-1} \tilde{c}_i \|\delta u^{n+1-i}\|^2,
\end{aligned}
\end{equation}
where $\tilde{c}_i$ is defined in \eqref{eq:tilde-c} by collecting the coefficients of $\|\delta u^{n+1-i}\|^2$ in the estimates for $I_3$ and $I_4$. We rewrite
\begin{equation}
\label{eq:thm1:14}
\begin{aligned}
\sum_{i=0}^{k-1}  \tilde{c}_i \|\delta u^{n+1-i}\|^2 & =  \sum_{i=1}^{k-1}  \tilde{c}_i \|\delta u^{n+1-i}\|^2  + \tilde{c}_0 \|\delta u^{n+1}\|^2 \\
& = \sum_{i=1}^{k-1}  \hat{c}_i \| \delta u^{n+1-i}\|^2 - \sum_{i=1}^{k-1}  \hat{c}_i \| \delta u^{n+2-i}\|^2 + \hat{c}_0 \|\delta u^{n+1}\|^2, \\
\end{aligned}
\end{equation}
where $\hat{c}_i = \sum_{j=i}^{k-1} \tilde{c}_j$ for $0\leq i \leq k-1$. Substituting \eqref{eq:thm1:14} into \eqref{eq:thm1:13}, and using the definition of $E_G^n$ in \eqref{eq:EG}, we have
\begin{equation}
\label{eq:thm1:18}
\begin{aligned}
& E_G^{n+1}-E_G^n \\
& \leq \frac{1}{\tau}  (\boldsymbol{w}_{n+1},  \mathcal{M}^{-1} \boldsymbol{w}_{n+1})_{U_a} - (\tilde{\boldsymbol{w}}_{n+1}, \mathcal{L} \tilde{\boldsymbol{w}}_{n+1})_{U_B} 
+  \frac{\ell_f \hat{c}_0 }{2} \|\delta u^{n+1}\|^2 \\
& \leq -\frac{\alpha}{\tau}\left\|(-\mathcal{M})^{-1 / 2} \delta u^{n+1}\right\|^2 -\beta\left\|\mathcal{L}^{1 / 2} \delta u^{n+1}\right\|^2  +  \frac{\ell_f \hat{c}_0 }{2}  \|\delta u^{n+1}\|^2. 
\end{aligned}
\end{equation}

In the case of $\eta = 1$, \eqref{eq:l2-control} gives $\|\delta u^{n+1}\|\leq \zeta \| (-\mathcal{M})^{-1/2}  \delta u^{n+1} \|$. 
It is straightforward to verify that 
$E_{G}^{n+1}\leq E_G^n$
when $\tau$ satisfies \eqref{eq:tau-max}. 
In the case of $0<\eta<1$, we apply Young's inequality $ab\leq \frac{a^p}{p}+\frac{b^q}{q}$ with $p=\frac{1}{\eta}$ and $q = \frac{1}{1-\eta}$. Replacing $a$ and $b$ by  $\xi^{\eta-1} a^{2\eta}$ and $\xi^{1-\eta}  b^{2-2\eta}$ respectively for any $\xi>0$, yields
\begin{equation}
a^{2\eta} b^{2-2\eta} \leq \eta \xi^{-\frac{1-\eta}{\eta}} a^2 + (1-\eta)\xi b^2.
\label{eq:thm1:15}
\end{equation}
Applying \eqref{eq:thm1:15} to \eqref{eq:l2-control}, we have
\begin{equation}
\label{eq:thm1:16}
\begin{aligned}
\|\delta u^{n+1}\|^2
 & \leq  \zeta^2 \| (-\mathcal{M})^{-1/2} \delta u^{n+1} \|^{2\eta} \| \mathcal{L}^{1/2} \delta u^{n+1} \|^{2-2\eta} \\
& \leq  \zeta^2 \left[\eta \xi^{-\frac{1-\eta}{\eta}} \| (-\mathcal{M})^{-1/2} \delta u^{n+1}\|^2
+ (1-\eta) \xi \| \mathcal{L}^{1/2} \delta u^{n+1}\|^2 \right].
\end{aligned}
\end{equation}
To ensure the decay of $E_G^n$ in \eqref{eq:thm1:18}, we then impose
\begin{equation}
\label{eq:thm1:17}
-\frac{\alpha}{\tau}+\frac{\ell_f \hat{c}_0}{2}  \zeta^2 \eta \xi^{-\frac{1-\eta}{\eta}} \leq 0 \quad \text {and}\quad -\beta+ \frac{\ell_f \hat{c}_0}{2} \zeta^2(1-\eta) \xi \leq 0 .
\end{equation}
By taking 
\[
\xi=\frac{2\beta}{  {\ell_f \hat{c}_0}  \zeta^2(1-\eta)},
\]
the second inequality of \eqref{eq:thm1:17} holds naturally and we then obtain the time-step restriction \eqref{eq:tau-max}.
\end{proof}

The following lemma shows the non-negativity of $E_G^n$.

\begin{lemma}[Sufficient conditions for non-negativity of the modified energy]
\label{lem:1}
Suppose that the assumptions of \cref{lemma:multiplier-LMM-energy} hold. Let $\hat{c}_i = \sum_{j=i}^{k-1}\tilde{c}_j$ for $0\le i\le k-1$, with $\tilde{c}_j$ defined in \eqref{eq:tilde-c}, and define the constants
\begin{equation}
\label{eq:bar-c}
\bar{c}_i = 
 \sum_{q=0}^{k-1} |\mu_q|
\sum_{m=0}^{i-1}\sum_{j=i}^{k-1}
|\mu_m|\,|\mu_j|\,(j-m).
\end{equation}
Let $c_{\min}\coloneqq \min_{1\le i\le k-1} (\hat{c}_i/2-\bar{c}_i)$.
Then, the modified energy $E_G^n$ defined in \eqref{eq:EG} satisfies $E_G^n \ge 0$ uniformly, provided 
\begin{equation}
\label{eq:bar-tau-max}
0 < \tau \le \bar{\tau}_{\max},\quad
\bar{\tau}_{\max} \coloneqq 
\left \{
\begin{aligned}
& +\infty, & c_{\min} \geq 0, \\
&  \frac{\lambda_a 4^{-\bar{\eta}}  \lambda_B^{\bar{\eta}}  }{( \ell_f  |c_{\min}| )^{1+\bar{\eta}} \eta (1-\eta)^{\bar{\eta}} \zeta^{2+2\bar{\eta}} }, & c_{\min} < 0,
\end{aligned}
\right .
\end{equation}
with $\lambda_a$ and $\lambda_B$ defined in \eqref{eq:lambda-ab}.
\end{lemma}

\begin{proof}
Let $\bar{u} = \sum_{j=0}^{k-1} \mu_j u^{n-j}$. 
Applying \eqref{eq:lip-F} with $u=u^{n-i}$ and $v=\bar{u}$, we have
\begin{equation}
\label{eq:lem1:1}
|(F(u^{n-i}) - F(\bar{u}), 1 )  - (f( \bar{u}), u^{n-i}-\bar{u}) |   \leq
\frac{\ell_f}{2} \| u^{n-i}-\bar{u}\|^2.
\end{equation}
Then
\begin{equation}
\label{eq:lem1:2}
\mu_i (F(u^{n-i}), 1)
\ge
\mu_i (F(\bar{u}), 1)
+ \mu_i (f(\bar{u}), u^{n-i} - \bar{u})
- |\mu_i| \frac{\ell_f}{2} \|u^{n-i} - \bar{u}\|^2.
\end{equation}
Summing \eqref{eq:lem1:2} over $i$ gives
\begin{equation}
\label{eq:lem1:3}
\begin{aligned}
 \sum_{i=0}^{k-1} \mu_i (F(u^{n-i}),1) & \geq \sum_{i=0}^{k-1} \left[ \mu_i (F(\bar{u}), 1) 
 +  \mu_i (f(\bar{u}), u^{n-i} - \bar{u}) 
 \right ]
 - \frac{\ell_f}{2} \sum_{i=0}^{k-1}  |\mu_i|  \|u^{n-i} - \bar{u}\|^2 \\
& \geq - \frac{\ell_f}{2} \sum_{i=0}^{k-1}  |\mu_i|  \|u^{n-i} - \bar{u}\|^2,
\end{aligned}
\end{equation}
where in the last inequality, we used the fact that $\sum_{i=0}^{k-1} \mu_i=1$, $(F(\bar{u}),1)\geq 0$, and $\sum_{i=0}^{k-1}  \mu_i (u^{n-i} - \bar{u}) = 0$. By the Cauchy--Schwarz inequality, we have
\begin{equation}
\label{eq:lem1:4}
\|u^{n-i}-\bar{u}\|^2
=
\left\|\sum_{j=0}^{k-1}\mu_j\bigl(u^{n-i}-u^{n-j}\bigr)\right\|^2
\leq \sum_{q=0}^{k-1} |\mu_q|\sum_{j=0}^{k-1}|\mu_j|\,\|u^{n-i}-u^{n-j}\|^2.
\end{equation}
Then
\begin{equation}
\label{eq:lem1:5}
\begin{aligned}
- \frac{\ell_f}{2} \sum_{i=0}^{k-1}  |\mu_i|  \|u^{n-i} - \bar{u}\|^2 & \geq - \frac{\ell_f}{2} \sum_{q=0}^{k-1} |\mu_q| \sum_{i=0}^{k-1}  \sum_{j=0}^{k-1}  |\mu_i|  \, |\mu_j|\,\|u^{n-i}-u^{n-j}\|^2 \\
& =-\,\ell_f \sum_{q=0}^{k-1} |\mu_q|
\sum_{0\le i<j\le k-1}
|\mu_i|\,|\mu_j|\,\|u^{n-i}-u^{n-j}\|^2.
\end{aligned}
\end{equation}
Note that $u^{n-i}-u^{n-j}
=\sum_{m=i+1}^{j}\delta u^{n+1-m}$.
Then we have $
\|u^{n-i}-u^{n-j}\|^2
\leq
(j-i)\sum_{m=i+1}^{j}\|\delta u^{n+1-m}\|^2$.
Substituting this into \eqref{eq:lem1:5} gives
\begin{equation}
\label{eq:lem1:7}
\begin{aligned}
- \frac{\ell_f}{2} \sum_{i=0}^{k-1}  |\mu_i|  \|u^{n-i} - \bar{u}\|^2
& \geq -\,\ell_f \sum_{q=0}^{k-1} |\mu_q|
\sum_{0\le i<j\le k-1}
|\mu_i|\,|\mu_j|\,(j-i)\sum_{m=i+1}^{j}\|\delta u^{n+1-m}\|^2 \\
% & =
% -\ell_f \sum_{m=1}^{k-1}
% \left[
%  \sum_{q=0}^{k-1} |\mu_q|
% \sum_{i=0}^{m-1}\sum_{j=m}^{k-1}
% |\mu_i|\,|\mu_j|\,(j-i)
% \right]
% \|\delta u^{n+1-m}\|^2 \\
& =
-\ell_f \sum_{m=1}^{k-1}
\bar{c}_m
\|\delta u^{n+1-m}\|^2,
\end{aligned}
\end{equation}
where $\bar{c}_m$ is defined in \eqref{eq:bar-c}. Substituting \eqref{eq:lem1:7} into \eqref{eq:lem1:3}, we obtain 
\begin{equation}
\label{eq:lem1:8}
\sum_{i=0}^{k-1} \mu_i (F(u^{n-i}),1)\geq
-\ell_f \sum_{m=1}^{k-1}
\bar{c}_m
\|\delta u^{n+1-m}\|^2.
\end{equation}

Substituting \eqref{eq:lem1:8} into the definition of $E^n_G$ in \eqref{eq:EG}, and using \eqref{eq:lambda-ab}, we have
\begin{equation}\small
\label{eq:lem1:9}
\begin{aligned}
E_G^n & = 
-\frac{1}{\tau} ( \boldsymbol{v}_{n},  \mathcal{M}^{-1} \boldsymbol{v}_{n})_{G_a}
 + (\boldsymbol{u}_{n}, \mathcal{L} \boldsymbol{u}_{n})_{G_B}  
 + \sum_{i=0}^{k-1} \mu_i (F(u^{n-i}) ,  1) 
 +  \frac{\ell_f}{2}   \sum_{i=1}^{k-1}  \hat{c}_i \| \delta u^{n+1-i}\|^2  \\
 & \geq \frac{1}{\tau} \lambda_a \|(-\mathcal{M})^{-1/2} \boldsymbol{v}_n\|^2 + \lambda_B \| \mathcal{L}^{1/2} \boldsymbol{u}_n\|^2 + \ell_f \sum_{i=1}^{k-1} \left(\frac{\hat{c}_i}{2} - \bar{c}_i \right)  \| \delta u^{n+1-i}\|^2.
\end{aligned}
\end{equation}
If $c_{\min}\geq 0$, then $\hat{c}_i/2 \geq \bar{c}_i$ for $i=1,\ldots,k-1$, and then $E^n_G\geq 0$ holds.

If ${c_{\min} < 0}$,
\begin{equation}
\label{eq:lem1:10}
\begin{aligned}
E_G^n  & \geq  \frac{1}{\tau} \lambda_a \|(-\mathcal{M})^{-1/2} \boldsymbol{v}_n\|^2 + \lambda_B \| \mathcal{L}^{1/2} \boldsymbol{u}_n\|^2 + \ell_f  {c_{\min}}   \sum_{i=1}^{k-1}  \| \delta u^{n+1-i}\|^2 \\
& \geq  \frac{1}{\tau} \lambda_a \|(-\mathcal{M})^{-1/2} \boldsymbol{v}_n\|^2 + \frac{\lambda_B}{4} \| \mathcal{L}^{1/2} \boldsymbol{v}_n\|^2 + \ell_f  {c_{\min}}   \| \boldsymbol{v}_n\|^2, \\
\end{aligned}
\end{equation}
where we use
\[
\begin{aligned}
\|\mathcal{L}^{1/2} \boldsymbol{v}_n\|^2 & = \sum_{i=0}^{k-2} \| \mathcal{L}^{1/2} \delta u^{n-i}\|^2 \\
& = \sum_{i=0}^{k-2} \left[ \| \mathcal{L}^{1/2} u^{n-i} \|^2 + \|\mathcal{L}^{1/2} u^{n-1-i}\|^2 - 2(\mathcal{L}^{1/2} u^{n-i}, \mathcal{L}^{1/2} u^{n-1-i}) \right ] \\
& \leq 2 \sum_{i=0}^{k-2} \left[  \|\mathcal{L}^{1/2} u^{n-i}\|^2 + \|\mathcal{L}^{1/2} u^{n-1-i}\|^2 \right ]
\leq 4 \|\mathcal{L}^{1/2} \boldsymbol{u}_n\|^2.
\end{aligned}
\]

In the case of $\eta = 1$, \eqref{eq:l2-control} gives $\|\delta u^{n+1-i}\|\leq \zeta \| (-\mathcal{M})^{-1/2}  \delta u^{n+1-i} \|$. Then $\|\boldsymbol{v}_n\| \leq \zeta \| (-\mathcal{M})^{-1/2}  \boldsymbol{v}_n \|$. 
It is straightforward to verify that 
$E_{G}^{n}\geq 0$
when $\tau$ satisfies \eqref{eq:bar-tau-max}. 
In the case of $0<\eta<1$, we replace $\delta u^{n+1}$ by $\delta u^{n+1-i}$ in \eqref{eq:thm1:16} and then obtain
\begin{equation}
\label{eq:lem1:11}
\|\delta u^{n+1-i}\|^2 
\leq \zeta^2 \left(
\eta \xi^{-\frac{1-\eta}{\eta}} \|(-\mathcal{M})^{-1/2} \delta u^{n+1-i}\|^2
+ (1-\eta)\xi \|\mathcal{L}^{1/2} \delta u^{n+1-i}\|^2
\right)
\end{equation}
for any $\xi>0$. Then
\begin{equation}
\label{eq:lem1:12}
\| \boldsymbol{v}_n \|^2 \leq \zeta^2 \left(
\eta \xi^{-\frac{1-\eta}{\eta}} \|(-\mathcal{M})^{-1/2} \boldsymbol{v}_n \|^2
+ (1-\eta)\xi \|\mathcal{L}^{1/2} \boldsymbol{v}_n \|^2
\right).
\end{equation}
To ensure that \eqref{eq:lem1:10} is lower bounded by 0, we then impose
\begin{equation}
\label{eq:lem1:13}
\frac{\lambda_a}{\tau}+  \ell_f  {c_{\min}}  \zeta^2 \eta \xi^{-\frac{1-\eta}{\eta}} \geq 0 \quad \text {and}\quad \frac{\lambda_B}{4} + \ell_f  {c_{\min}}  \zeta^2(1-\eta) \xi \geq 0 .
\end{equation}
By taking 
\[
\xi = \frac{\lambda_B}{ 4 \ell_f {|c_{\min}|} \zeta^2 (1-\eta)},
\]
the second inequality of \eqref{eq:lem1:13} holds naturally and we then obtain the time-step restriction \eqref{eq:bar-tau-max}.
\end{proof}

We now present a property of $G_B$ that will be used to obtain the consistency of $E_G^n$ with the original energy $E[u(t_n)]$.

\begin{lemma}
\label{lem:1gnu1}
Any real matrix $G_B$ satisfying the positive-definiteness conditions in
\cref{lemma:multiplier-LMM-energy} satisfies
\begin{equation}
\label{eq:lgnu1:1}
\boldsymbol{1}_k^\top G_B \boldsymbol{1}_k = \frac{1}{2} \sum_{i=0}^{k-1} \mu_i,
\end{equation}
where $\boldsymbol{1}_k = [1,\ldots,1]^\top\in\mathbb{R}^k$.
\end{lemma}
\begin{proof}
Multiplying \eqref{eq:UB} from the left by $\boldsymbol{1}_{k+1}^\top$ and right by $\boldsymbol{1}_{k+1}$ yields
\[
\boldsymbol{1}_{k+1}^\top U_B \boldsymbol{1}_{k+1} = 
\frac{1}{2}
\boldsymbol{1}_{k+1}^\top \left( \boldsymbol{\nu} (\boldsymbol{B}^{(k)})^\top  + \boldsymbol{B}^{(k)} \boldsymbol{\nu}^\top \right )\boldsymbol{1}_{k+1} - 
\boldsymbol{1}_{k}^\top G_B \boldsymbol{1}_{k}
+
\boldsymbol{1}_{k}^\top G_B \boldsymbol{1}_{k} = 0,
\]
where we have used the property $\boldsymbol{1}_{k+1}^\top \boldsymbol{\nu}  = 0$ that follows from the definition \eqref{eq:multiplier2}. Since $U_B$ is positive semi-definite, it follows that $U_B \boldsymbol{1}_{k+1} = \boldsymbol{0}_{k+1}$.
Then multiplying \eqref{eq:UB} on the right by $\boldsymbol{1}_{k+1}$ yields
\begin{equation}
\label{eq:lgnu1:2}
U_B \boldsymbol{1}_{k+1} = 
\frac{1}{2} \boldsymbol{\nu} (\boldsymbol{B}^{(k)})^\top \boldsymbol{1}_{k+1} 
- \begin{bmatrix} \boldsymbol{g} \\ 0 \end{bmatrix}
+ \begin{bmatrix} 0 \\ \boldsymbol{g} \end{bmatrix}
= \boldsymbol{0}_{k+1},
\end{equation}
where $\boldsymbol{g} \coloneqq G_B \boldsymbol{1}_k \in \mathbb{R}^k$. Writing $\boldsymbol{g} = [g_1,\ldots,g_k]^\top$, by \eqref{cond:normalization}, we have $(\boldsymbol{B}^{(k)})^\top \boldsymbol{1}_{k+1}=1$, and obtain from \eqref{eq:lgnu1:2} the elementwise recurrence: $g_1 = \nu_0/2 = \mu_0/2$ and $g_{j+1} - g_j = \nu_j/2$ for $1 \le j \le k-1$. 
By summing this recurrence, we readily obtain
\[
g_i = \frac{1}{2} \sum_{j=0}^{i-1} \nu_j = \frac{\mu_{i-1}}{2} \quad  \text{for}\quad i=1,\ldots,k. 
\]
That is, $\boldsymbol{g} = \frac{1}{2} [\mu_0, \mu_1, \dots, \mu_{k-1}]^\top$. 
Therefore $\boldsymbol{1}_k^\top G_B \boldsymbol{1}_k
=  \boldsymbol{1}_k^\top \boldsymbol{g}
= \frac{1}{2} \sum_{i=0}^{k-1} \mu_i$.
\end{proof}

\begin{remark}[Consistency of $E_G^n$ with the original energy]\label{rem:consistency}
Assume that $u(t)$ is sufficiently smooth in time, and let $u^{n-i}=u(t_{n-i})$. Then $\delta u^{n-j}=u(t_{n-j})-u(t_{n-j-1})=\mathcal O(\tau)$ for $j=0,\dots,k-2$,
so that $\boldsymbol v_n=\mathcal O(\tau)$. It follows from \eqref{eq:EG} that
\[
E_G^n
=
(\boldsymbol{u}_n, \mathcal{L}\boldsymbol{u}_n)_{G_B}
+\sum_{i=0}^{k-1}\mu_i (F(u^{n-i}),1)
+\mathcal O(\tau).
\]
As $\tau\to0$ with $t_n$ fixed, we have $\boldsymbol{u}_n\to u(t_n)\boldsymbol{1}_k$ and $u^{n-i}\to u(t_n)$. Hence, by \cref{lem:1gnu1},
\[
E_G^n \to
(\boldsymbol{1}_k^\top G_B \boldsymbol{1}_k)(u(t_n),\mathcal{L}u(t_n))
+\sum_{i=0}^{k-1}\mu_i (F(u(t_n)),1)
=E[u(t_n)].
\]
Thus, for any $G_B$ satisfying \eqref{eq:lgnu1:1}, the modified
energy $E_G^n$ is consistent with the continuous energy $E[u]$.
\end{remark}

% \subsection{Modified energy dissipation theorem} \label{subsec:main-result}

We are now ready to state and prove our main result on the modified energy dissipation of the IMEX-LMMs for gradient flows.

\begin{theorem}
\label{thm:main}
% {\blue Assume that the nonlinearity in \eqref{model:phase-field-eq} is globally Lipschitz with constant $\ell_f$ as in \eqref{eq:lip-f}, and that \eqref{eq:l2-control} holds for some constants $\zeta>0$ and $0<\eta\le 1$.}
Assume that the gradient flow \eqref{model:phase-field-eq} has a Lipschitz continuous nonlinearity satisfying \eqref{eq:lip-f}, and that there exist constants $\zeta>0$ and $0<\eta\le 1$ such that \eqref{eq:l2-control} holds.
Consider the $k$-step IMEX-LMM \eqref{scheme:lmm-diff-form}, and let $\boldsymbol{a}^{(k)}$, $\boldsymbol{B}^{(k)}$, and $\hat{\boldsymbol{b}}^{(k)}$ be the coefficient vectors defined in \eqref{eq:coef-vecs}. Let the multiplier in \eqref{eq:multiplier1} be specified by the vectors $\boldsymbol{\mu}$ and $\boldsymbol{\nu}$ defined in \eqref{eq:multiplier2}. 

Let $M(z;\boldsymbol{a}^{(k)})$, $M(z;\boldsymbol{B}^{(k)})$ and $M(z;\boldsymbol{\mu})$ be the generating polynomials associated with $\boldsymbol{a}^{(k)}$, $\boldsymbol{B}^{(k)}$ and $\boldsymbol{\mu}$, defined by \eqref{eq:genpoly}. Suppose that $\deg M(z;\boldsymbol{a}^{(k)})=\deg M(z;\boldsymbol{\mu}) = k-1$ and $\deg M(z;\boldsymbol{B}^{(k)}) = k$, 
that $M(z;\boldsymbol{a}^{(k)})$ and $M(z;\boldsymbol{B}^{(k)})$ are relatively prime to $M(z;\boldsymbol{\mu})$ and $M(z;\boldsymbol{\nu})$, respectively, and that $
M(z;\boldsymbol{\mu})\neq 0$ for $|z|>1$. 
Assume further that there exist constants $\alpha,\beta>0$ such that for any $\theta\in[0,\pi]$,
\begin{equation}
\label{ieq:positivity}
    \operatorname{Re}\bigl\{
M(\mathrm{e}^{\mathrm{i}\theta};\boldsymbol{a}^{(k)})
M(\mathrm{e}^{-\mathrm{i}\theta};\boldsymbol{\mu})
\bigr\}\ge \alpha
,\quad
\operatorname{Re}\bigl\{
M(\mathrm{e}^{\mathrm{i}\theta};\boldsymbol{B}^{(k)})
M(\mathrm{e}^{-\mathrm{i}\theta};\boldsymbol{\nu})
\bigr\}\ge 2\beta(1-\cos\theta).
\end{equation}
Then there exist positive definite matrices
$G_a\in\mathbb{R}^{(k-1)\times(k-1)}$ and
$G_B\in\mathbb{R}^{k\times k}$ satisfying the positive-definiteness conditions
\eqref{eq:pdc} in \cref{lemma:multiplier-LMM-energy}. For these matrices, the modified energy $E_G^n$
defined by \eqref{eq:EG} satisfies $E_G^{n+1}\le E_G^n$ 
whenever  $0<\tau\leq \tau_{\max}$, where $\tau_{\max}$ is defined in \eqref{eq:tau-max}. 

Moreover, under the additional assumptions of \cref{lem:1}, the modified energy $E_G^n$ is dissipative and nonnegative under restriction $0<\tau\le \min\{\tau_{\max},\bar{\tau}_{\max}\}$, where $\bar{\tau}_{\max}$ is defined in \eqref{eq:bar-tau-max}.

% Moreover, under the additional assumptions of \cref{lem:1}, the modified energy 
% is uniformly nonnegative. More precisely, condition {\rm (i)} of \cref{lem:1} yields 
% uniform non-negativity without any additional time-step restriction, and hence
% $E_G^n$ is both dissipative and uniformly nonnegative whenever
% $0<\tau\leq \tau_{\max}$. If condition {\rm (ii)} of \cref{lem:1} holds, then $E_G^n$ is 
% both dissipative and uniformly nonnegative whenever $0<\tau\le \min\{\tau_{\max},\bar{\tau}_{\max}\}$, where $\bar{\tau}_{\max}$ is defined in \eqref{eq:bar-tau-max}.

% In addition, $E_G^n$ is consistent with the original energy.
\end{theorem}

% {\blue 
% The coprimality assumption excludes the simple multiplier $\boldsymbol\mu=[1,0,\ldots,0]^\top$ for IMEX-BDF schemes, in which case $M(z;\boldsymbol B^{(k)})=z^k$ and $M(z;\boldsymbol\nu)=(z-1)z^{k-1}$ share the common divisor $z^{k-1}$. This monomial degeneracy can be removed by canceling the common divisor in \cref{thm:equivalence}; see \cref{rem:degenerate-case} for details.
% }

% {\blue
% Combined with \cref{lem:1} below, this implies that the modified energy $E_G^n$ is both nonnegative and monotonically decaying, provided that $\tau\leq\min(\tau_{\max},\bar{\tau}_{\max})$, where $\bar{\tau}_{\max}$ is given in \eqref{eq:bar-tau-max}. When the first case of \cref{lem:1} holds (i.e., $\hat{c}_i\ge 2\bar{c}_i$ for all $1\le i\le k-1$), no additional restriction beyond $\tau\leq\tau_{\max}$ is required.
% }

\begin{proof}
% {\blue By \cref{lemma:multiplier-LMM-energy}, it suffices to verify the positive-definiteness conditions \eqref{eq:pdc}. We apply \cref{thm:equivalence} twice: once with $(\rho,\sigma,\lambda,\gamma)=(M(\cdot;\boldsymbol{a}^{(k)}),M(\cdot;\boldsymbol{\mu}),z^{k-1},\alpha)$ to handle $U_a$, and once with $(\rho,\sigma,\lambda,\gamma)=(M(\cdot;\boldsymbol{B}^{(k)}),M(\cdot;\boldsymbol{\nu}),z^k-z^{k-1},\beta)$ to handle $U_B$. The specific form $\lambda(z)=z^k-z^{k-1}$ is chosen so that $|\lambda(\mathrm{e}^{\mathrm{i}\theta})|^2 = 2(1-\cos\theta)$, matching the assumed bound on the $U_B$-side.}
We prove the dissipation of $E_G^n$ by verifying the positive-definiteness conditions \eqref{eq:pdc} in \cref{lemma:multiplier-LMM-energy}. This is done by applying \cref{thm:equivalence} to $U_a$ and $U_B$, respectively.

For $U_a$, let $\rho(z)=M(z;\boldsymbol a^{(k)})$,
$\sigma(z)=M(z;\boldsymbol\mu)$, $\lambda(z)=z^{k-1}$, and
$\gamma=\alpha$. The degree and coprimality assumptions give the corresponding
assumptions of \cref{thm:equivalence}. Moreover, $\sigma(z)\neq 0$ for
$|z|>1$ according to the assumption. Since the coefficients are real, the inequality assumed for
$\theta\in[0,\pi]$ yields condition (iii)(b) of \cref{thm:equivalence} on the
whole unit circle. If $\sigma(z_0)=0$ for some $|z_0|=1$, then
\[
0=\operatorname{Re}\{\rho(z_0)\sigma(\bar z_0)\}
\ge \alpha |\lambda(z_0)|^2=\alpha,
\]
which is impossible. Therefore, condition (iii) of \cref{thm:equivalence} holds.
Hence, by \textnormal{(iii)} $\Longleftrightarrow$ \textnormal{(v)}, there exist a positive definite matrix $G$ and a vector $\boldsymbol{q}$ such that \eqref{eq:equi-thm:cond5} is satisfied. Setting $G_a=G/2$ and substituting \eqref{eq:equi-thm:cond5} into \eqref{eq:Ua}, we obtain $
U_a=\frac12\boldsymbol{q}\boldsymbol{q}^\top+\alpha \boldsymbol{e}_1\boldsymbol{e}_1^\top$, where $\boldsymbol{e}_1=[1,0,\ldots,0]^\top \in \mathbb{R}^k$, 
and then the $U_a$-condition in \eqref{eq:pdc} follows.

For $U_B$, let $\rho(z)=M(z;\boldsymbol{B}^{(k)})$, $\sigma(z)=M(z;\boldsymbol{\nu})=(z-1)M(z;\boldsymbol{\mu})$, $\lambda(z)=z^k-z^{k-1}$, 
and $\gamma=\beta$.
We verify condition \textnormal{(iii)} of \cref{thm:equivalence}. Condition \textnormal{(iii)(a)} follows from $M(z;\boldsymbol{\mu})\neq 0$ for $|z|>1$. On $|z|=1$, writing $z=\mathrm{e}^{\mathrm{i}\theta}$ gives
\[
|\lambda(z)|^2=|z^k-z^{k-1}|^2=|z-1|^2=2(1-\cos\theta),
\]
and thus the assumed inequality gives condition \textnormal{(iii)(b)}.
It remains to verify condition \textnormal{(iii)(c)}. If $z_*$ is a root of $\sigma(z)$ on $|z|=1$ and $z_*\neq 1$, then 
$|\lambda(z_*)|^2=|z_*-1|^2>0$. Condition
(iii)(b) would then imply
\[
0=\operatorname{Re}\{\rho(z_*)\sigma(\bar z_*)\}
\ge \beta|\lambda(z_*)|^2>0,
\]
a contradiction. Hence, the only possible unit-circle root of $\sigma(z)$ is $z_*=1$.
Since
\[
\sigma(z)=(z-1)M(z;\boldsymbol\mu) \quad \text{and}\quad 
M(1;\boldsymbol\mu)=\sum_{i=0}^{k-1}\mu_i=1,
\]
this root is simple and $\sigma'(1)=1$. By \eqref{cond:normalization},
$\rho(1)=M(1;\boldsymbol B^{(k)})=1$, and hence
\[
\frac{\rho(1)}{\sigma'(1)}=1>0.
\]
Thus, condition (iii)(c) also holds.
Applying again \textnormal{(iii)} $\Longleftrightarrow$ \textnormal{(v)}, we obtain a positive definite matrix $G$ and a vector $\boldsymbol{q}$ such that \eqref{eq:equi-thm:cond5} holds. Setting $G_B=G/2$ and substituting \eqref{eq:equi-thm:cond5} into  \eqref{eq:UB}, we get
\[
U_B=\frac12\boldsymbol{q}\boldsymbol{q}^\top+\beta \boldsymbol{\lambda}\boldsymbol{\lambda}^\top,
\]
where $\boldsymbol{\lambda}=[1,-1,0,\ldots,0]^\top\in\mathbb{R}^{k+1}$, 
which yields the $U_B$-condition in \eqref{eq:pdc}.

Therefore, all assumptions of \cref{lemma:multiplier-LMM-energy} are satisfied, and the conclusion $E_G^{n+1}\le E_G^n$
under the time-step restriction \eqref{eq:tau-max} follows from \cref{lemma:multiplier-LMM-energy}.
\end{proof}

\begin{remark}
The coprimality assumption in \cref{thm:main} excludes some useful degenerate
cases. For example, for the $k$-step IMEX-BDF scheme with the simple multiplier
$\boldsymbol{\mu}=[1,0,\ldots,0]^\top$ (see \cite{QuanWangWangXuPre}), we have
\[
M(z;\boldsymbol B^{(k)})=z^k\quad \text{and}\quad 
M(z;\boldsymbol\nu)=(z-1)z^{k-1}.
\]
Thus, the pair $\bigl(M(z;\boldsymbol B^{(k)}),M(z;\boldsymbol\nu)\bigr)$ is not
coprime, since the two polynomials share the common divisor $z^{k-1}$. This case is
therefore not covered by the strict statement of \cref{thm:main}.
Nevertheless, in the verification of the $U_B$-condition, the same common divisor also appears in
\[
\lambda(z)=z^k-z^{k-1}=z^{k-1}(z-1).
\]
After ignoring the common divisor $z^{k-1}$, one obtains the reduced triple
\[
\rho(z)=z,\qquad \sigma(z)=z-1,\qquad \lambda(z)=z-1,
\]
to which \cref{thm:equivalence} applies. In this reduced case, one may take
\[
\beta=\frac12,\qquad
G_B=\frac12\operatorname{diag}(1,0,\ldots,0)\succeq 0,
\]
which gives
\[
(\boldsymbol u_n,\mathcal{L}\boldsymbol u_n)_{G_B}
=\frac{1}{2}(u^n,\mathcal{L}u^n).
\]
Thus, the required $U_B$ inequality is still obtained, although the embedded matrix
$G_B$ is only positive semi-definite rather than positive definite. For this reason,
we keep \cref{thm:main} in the strict coprime and positive definite form, and treat
such monomial degenerate cases separately.

More generally, if
\[
\gcd\bigl(M(z;\boldsymbol B^{(k)}),M(z;\boldsymbol\nu)\bigr)=z^m
\]
for some $1\le m\le k-1$, where $\gcd(p,q)$ denotes the greatest common divisor of $p$ and $q$, then the same reduction may be applied after ignoring the common divisor $z^m$, provided that the reduced polynomial triple satisfies the
assumptions of \cref{thm:equivalence}.
\end{remark}

% {\blue 
% \begin{remark}\label{rem:degenerate-case}
% The coprimality assumption in \Cref{thm:main} excludes the simple multiplier $\boldsymbol{\mu}=[1,0,\ldots,0]^\top$ for the $k$-step IMEX-BDF scheme, {\blue since $M(z;\boldsymbol{B}^{(k)})=z^k$ and $M(z;\boldsymbol{\nu})=(z-1)z^{k-1}$ share the common divisor $z^{k-1}$.}
% % since $M(z;\boldsymbol{B}^{(k)})=z^k$ and  $M(z;\boldsymbol{\nu})=(z-1)z^{k-1}$, and hence $(M(z;\boldsymbol{B}^{(k)})$ and $M(z;\boldsymbol{\nu}))$ {\red are relatively prime}.
% This is a degenerate case of \Cref{thm:main}. Indeed, in the proof of
% \Cref{thm:main}, the polynomial $\lambda(z)=z^{k-1}(z-1)$ contains the same common divisor. After canceling $z^{k-1}$, one obtains the reduced
% triple
% \[
% \rho(z)=z,\quad \sigma(z)=z-1,\quad \lambda(z)=z-1,
% \]
% to which \Cref{thm:equivalence} applies. In particular, one may take
% \[
% \beta=\frac{1}{2},
% \qquad
% G_B=\frac12\,\operatorname{diag}(1,0,\ldots,0)\succeq 0,
% \]
% so that
% \[
% \boldsymbol{u}_n^\top G_B( \mathcal{L}\boldsymbol{u}_n)=\frac{1}{2}\,(\boldsymbol{u}^n,\mathcal{L}\boldsymbol{u}^n).
% \]
% More generally, the same reduction applies whenever
% \[
% \gcd(M(z;\boldsymbol{B}^{(k)}),M(z;\boldsymbol{\nu}))=z^m
% \]
% for some $1\le m\le k-1$, where $\gcd(p,q)$ denotes the greatest common divisor of $p$ and $q$. We therefore keep \Cref{thm:main} in the strict coprime form and treat such monomial degenerate cases separately.
% \end{remark}
% }

\section{Feasibility problems and high-order IMEX-LMMs}\label{sec:feasibility}

\Cref{thm:main} ensures the preservation of modified energy dissipation by degree, Schur stability, coprimality, and positivity conditions on the generating polynomials of the scheme and the multiplier.
In this section, we first  derive an affine representation \eqref{eq:matrix-parameterization} of the coefficients of $k$th-order IMEX-LMMs satisfying \eqref{cond:order-lmm}--\eqref{cond:normalization}, and then impose the conditions of \Cref{thm:main} on this representation to formulate a feasibility problem over the coefficients of the scheme and the multiplier.
This feasibility problem is used to find multipliers for establishing the modified energy dissipation of IMEX-BDF6 and IMEX-WSBDF7 and to construct a new energy-dissipative IMEX-LMM8.

\subsection{Feasibility problems}
\label{subsec:parameterization}

We use the implicit coefficient vector
\[
    \boldsymbol B^{(k)}=[B_0^{(k)},\ldots,B_k^{(k)}]^\top
\]
as free parameters under the normalization restriction \eqref{cond:normalization}.
For a prescribed $\boldsymbol B^{(k)}$, the order conditions \eqref{cond:order-lmm} determine the remaining coefficients, $\boldsymbol A^{(k)}$ and $\hat{\boldsymbol B}^{(k)}$, in the IMEX-LMM \eqref{eq:model-lmm2}. More precisely, the equations in \eqref{cond:order-lmm} give
\[
    W_1\boldsymbol A^{(k)}
    =
    \begin{bmatrix}
        0\\
        D_k^{-1}[I_k,\boldsymbol{0}_{k}]W_1\boldsymbol B^{(k)}
    \end{bmatrix},
    \quad
    W_3\hat{\boldsymbol B}^{(k)}
    =
    [I_k,\boldsymbol{0}_{k}]W_1\boldsymbol B^{(k)} ,
\]
where
\[
D_k=\operatorname{diag}\braB{1,\frac12,\ldots,\frac1k},
\quad
W_1=W(0,-1,\ldots,-k)^\top,\quad
W_3=W(-1,\ldots,-k)^\top,
\]
and $W(x_0,\ldots,x_n)_{i+1,j+1}=x_i^j$ for $0\le i,j\le n$.
Using \eqref{def:coefs-a-b-c}, we obtain
\begin{equation}
\label{eq:matrix-parameterization}
\begin{aligned}
\boldsymbol a^{(k)}
&=
[E_k,\boldsymbol{0}_{k}]W_1^{-1}
\begin{bmatrix}
0\\
D_k^{-1}[I_k,\boldsymbol{0}_{k}]W_1\boldsymbol B^{(k)}
\end{bmatrix},\;
\hat{\boldsymbol b}^{(k)}
=
E_kW_3^{-1}[I_k,\boldsymbol{0}_{k}]W_1\boldsymbol B^{(k)}
-\boldsymbol 1_k .
\end{aligned}
\end{equation}
Here, $E_k\in\mathbb R^{k\times k}$ is the lower triangular matrix with $(E_k)_{ij}=1$ for $i\ge j$ and $(E_k)_{ij}=0$ otherwise. Thus, $\boldsymbol a^{(k)}$ and $\hat{\boldsymbol b}^{(k)}$, and equivalently $\boldsymbol A^{(k)}$ and $\hat{\boldsymbol B}^{(k)}$, are determined linearly by $\boldsymbol B^{(k)}$.

Define
\[
\begin{aligned}
    P_a(\cos\theta;\boldsymbol B^{(k)},\boldsymbol\mu) &
    \coloneqq
    \operatorname{Re}\bigl\{
    M(\mathrm e^{\mathrm i\theta};\boldsymbol a^{(k)})
    M(\mathrm e^{-\mathrm i\theta};\boldsymbol\mu)
    \bigr\} \qquad\text{for $\theta\in [0,\pi]$,} \\
    P_B(\cos\theta;\boldsymbol B^{(k)},\boldsymbol\mu) &
    \coloneqq
    \frac{
    \operatorname{Re}\bigl\{
    M(\mathrm e^{\mathrm i\theta};\boldsymbol B^{(k)})
    M(\mathrm e^{-\mathrm i\theta};\boldsymbol\nu)
    \bigr\}
    }{2(1-\cos\theta)} \qquad\text{for $\theta\in (0,\pi]$},
\end{aligned}
\]
where the coefficients $a_j^{(k)}$ in $P_a(\cos\theta;\boldsymbol B^{(k)},\boldsymbol\mu)$ are determined by \eqref{eq:matrix-parameterization}.
With the variable transformation $x=\cos\theta$, they are equivalent to
\[
\begin{aligned}
    P_a(x;\boldsymbol B^{(k)},\boldsymbol\mu) &
    =
    \sum_{j=0}^{k-1} a_j^{(k)}\mu_j
    +
    \sum_{m=1}^{k-1}
    \sum_{j=0}^{k-1-m}
    \bigl(
    a_j^{(k)}\mu_{j+m}
    +
    a_{j+m}^{(k)}\mu_j
    \bigr)T_m(x), \\
    P_B(x;\boldsymbol B^{(k)},\boldsymbol\mu) &
    =
    \frac{1}{2(1-x)}
    \biggl[
    \sum_{j=0}^{k} B_j^{(k)}\nu_j
    +
    \sum_{m=1}^{k}
    \sum_{j=0}^{k-m}
    \bigl(
    B_j^{(k)}\nu_{j+m}
    +
    B_{j+m}^{(k)}\nu_j
    \bigr)T_m(x)
    \biggr],
\end{aligned}
\]
where $T_m$ denotes the $m$-th Chebyshev polynomial of the first kind, defined by
\[
T_m(\cos\theta)=\cos(m\theta).
\]
The continuous extension at $x=1$ in the definition of $P_B$ is well-defined because $M(1;\boldsymbol\nu)=0$ (here, $M$ is the generating polynomial \eqref{eq:genpoly} associated with $\boldsymbol\nu$), so the numerator in the displayed quotient for $P_B$ vanishes at $x=1$. Thus, $P_a$ and $P_B$ are polynomials in $x$, and the positivity conditions \eqref{ieq:positivity} are equivalent to
\[
    P_a(x;\boldsymbol B^{(k)},\boldsymbol\mu)\ge \alpha,
    \qquad
    P_B(x;\boldsymbol B^{(k)},\boldsymbol\mu)\ge \beta,
    \qquad x\in[-1,1].
\]

The following feasibility problem formulates the four verifiable conditions in \Cref{thm:main} in terms of the scheme coefficients $\boldsymbol B^{(k)}$ and the multiplier coefficients $\boldsymbol \mu$.
\begin{enumerate}[label=\textnormal{(FP)}, ref=FP, leftmargin=*]
    \item \label{FP} Find $\boldsymbol B^{(k)}\in\mathbb{R}^{k+1}$, $\boldsymbol{\mu}\in\mathbb{R}^k$, and constants $\alpha,\beta>0$ such that $\sum_{i=0}^{k} B^{(k)}_i = 1$, $\sum_{i=0}^{k-1} \mu_i = 1$ and the following conditions hold:
    \begin{itemize}
        \item[\textnormal{(a)}] Degree condition: The generating polynomials attain their polynomial degrees
        \[
            \deg M(z; \boldsymbol{a}^{(k)}) = \deg M(z; \boldsymbol{\mu}) = k-1, \quad \deg M(z; \boldsymbol{B}^{(k)}) = k,
        \]
        which is equivalent to $a_0^{(k)}\ne 0$, $\mu_0\ne 0$, and $B_0^{(k)}\ne 0$.
        \item[\textnormal{(b)}] Schur stability: All roots of the multiplier polynomial $M(z;\boldsymbol{\mu})$ lie in the open unit disk (see, e.g., \cite{Gargantini1971}); equivalently, $M(z;\boldsymbol{\mu})$ is Schur stable.
        \item[\textnormal{(c)}] Coprimality: The pairs $(M(z;\boldsymbol{a}^{(k)}),M(z;\boldsymbol{\mu}))$ and 
        $(M(z;\boldsymbol{B}^{(k)}),M(z;\boldsymbol{\nu}))$ are coprime, respectively.
        \item[\textnormal{(d)}] Positivity: The following positivity inequalities hold for $x\in [-1, 1]$:
        \begin{equation}\label{eq:FP-poly}
            P_a(x;\boldsymbol B^{(k)},\boldsymbol\mu)\ge \alpha,
            \qquad
            P_B(x;\boldsymbol B^{(k)},\boldsymbol\mu)\ge \beta.
        \end{equation}
    \end{itemize}
     Here, $\boldsymbol{a}^{(k)}$ is determined by \eqref{eq:matrix-parameterization}, and $\boldsymbol{\nu}$ is determined by \eqref{eq:multiplier2}.
\end{enumerate}
\noindent Fixing the scheme gives the following reduced problem:
\begin{enumerate}[label=\textnormal{(FP-Multiplier)}, ref=FP-Multiplier, leftmargin=*]
    \item \label{FP-Multiplier} Given $\boldsymbol B^{(k)}\in\mathbb{R}^{k+1}$, find $\boldsymbol{\mu}\in\mathbb{R}^k$ and $\alpha,\beta>0$ satisfying \eqref{FP}.
\end{enumerate}

For fixed $\boldsymbol B^{(k)}$, the inequalities in \eqref{eq:FP-poly} are linear in $(\boldsymbol\mu,\alpha,\beta)$.
% For fixed $\boldsymbol\mu$, they are linear in $(\boldsymbol B^{(k)},\alpha,\beta)$.
Consequently, after discretizing $[0,\pi]$ by finite grids $\Theta_a\subset[0,\pi]$ and $\Theta_B\subset(0,\pi]$, the two inequalities in \eqref{eq:FP-poly} become finitely many linear constraints, leading to the following LP for \eqref{FP-Multiplier}:
\begin{equation}
\label{eq:LP-multiplier}
\begin{array}{ll}
\textnormal{maximize}_{\boldsymbol\mu,\gamma_\ast}
    & \gamma_\ast \\[0.5ex]
\textnormal{subject to}
    & P_a(\cos\theta;{\boldsymbol B}^{(k)},\boldsymbol\mu)\ge \gamma_\ast,
      \quad \theta\in\Theta_a,\\
    & P_B(\cos\theta;{\boldsymbol B}^{(k)},\boldsymbol\mu)\ge \gamma_\ast,
      \quad \theta\in\Theta_B,\\
    & 
      \gamma_\ast\ge0,\quad \sum_{i=0}^{k-1}\mu_i=1 .
\end{array}
\end{equation}
% For \eqref{FP-Scheme}, with $\boldsymbol\mu=\overline{\boldsymbol\mu}$ fixed, the LP is
% \begin{equation}
% \label{eq:LP-scheme}
% \begin{array}{ll}
% \textnormal{maximize}_{\boldsymbol B^{(k)},\alpha,\beta,\gamma_\ast}
%     & \gamma_\ast \\[0.5ex]
% \textnormal{subject to}
%     & P_a(\cos\theta;\boldsymbol B^{(k)},\overline{\boldsymbol\mu})\ge \alpha,
%       \quad \theta\in\Theta_a,\\
%     & P_B(\cos\theta;\boldsymbol B^{(k)},\overline{\boldsymbol\mu})\ge \beta,
%       \quad \theta\in\Theta_B,\\
%     & \gamma_\ast\le \alpha,\qquad \gamma_\ast\le \beta,\qquad
%       \alpha,\beta,\gamma_\ast\ge0, \\
%     & \sum_{i=0}^{k}B_i^{(k)}=1.
% \end{array}
% \end{equation}
% The use of $\gamma_\ast$ follows the standard auxiliary-variable reformulation of a max-min objective; see, e.g., 
% \cite[Section~4.1.3, Eq.~(4.11)]{boyd2004convex}.
% \cite[Section~4.1.3, Eq.~(4.11)]{boyd2004convex}.
In our implementation, this LP is solved using MATLAB's \texttt{linprog} solver to generate multiplier candidates for \eqref{FP-Multiplier}.
The selected candidate is then verified against the original feasibility conditions in \eqref{FP} as follows.
\begin{itemize}
    \item[(i)] The degree condition is obtained directly by $a_0^{(k)}\ne0$, $\mu_0\ne0$, and $B_0^{(k)}\ne0$.
    \item[(ii)] The Schur stability of $M(z;\boldsymbol\mu)$ is verified by the Schur--Cohn criterion \cite[Theorem 5.1]{elaydi2005}.
    \item[(iii)] The coprimality of $(M(z;\boldsymbol a^{(k)}), M(z;\boldsymbol\mu))$ and $(M(z;\boldsymbol B^{(k)}), M(z;\boldsymbol\nu))$ is checked by computing the corresponding resultants, which rule out common roots when the resultants are nonzero; see \cite[Theorem 4.1]{sturmfels2002}.
    \item[(iv)] The positivity conditions in \eqref{eq:FP-poly} are certified by applying Sturm's theorem to the shifted polynomials $P_a-\alpha$ and $P_B-\beta$ to count their real roots in $[-1,1]$, followed by a sign check at one point of the interval; see, e.g., \cite[Theorem 1.4]{sturmfels2002}.
    % \item[(iv)] The positivity condition \eqref{eq:FP-D1} is reduced by the variable transformation $x=\cos\theta$ to a polynomial positivity problem on $[-1,1]$.
    % For \eqref{eq:FP-D2}, the same transformation gives an inequality of the form $R_B(x)\ge 2\beta(1-x)$.
    % Since $M(1;\boldsymbol\nu)=0$, we have $R_B(1)=0$, and hence $R_B(x)$ is divisible by $1-x$.
    % Then $R_B(x)/(2(1-x))$ with its continuous extension at $x=1$ is a polynomial, and \eqref{eq:FP-D2} is likewise reduced to a polynomial positivity problem on $[-1,1]$. These polynomial positivity problems are certified by Sturm's theorem, which gives the number of real roots in an interval; see, e.g., \cite[Theorem 1.4]{sturmfels2002}.
\end{itemize}
% \textcolor{red}{The matrices $G_a$ and $G_B$ are obtained from \cref{rem:determine-G}.}

\subsection{Multipliers for IMEX-BDF6 and IMEX-WSBDF7 schemes}
\label{subsec:energy-stability-6-7}

In this subsection, we solve the finite-grid LP \eqref{eq:LP-multiplier} to obtain a set of multiplier candidates for the IMEX-BDF6 and IMEX-WSBDF7 schemes, respectively.
Then these candidates are verified against the original feasibility conditions in \eqref{FP}, which ensures modified energy dissipation according to \Cref{thm:main}.

\begin{proposition}
\label{prop:bdf6-feasible}
The IMEX-BDF6 scheme defined by
\begin{equation*}
\begin{aligned}
    \boldsymbol{A}^{(6)} &= \Bigl[
        \frac{49}{20},-6,\frac{15}{2},-\frac{20}{3},
        \frac{15}{4},-\frac{6}{5},\frac{1}{6}
    \Bigr]^\top,\\
    \boldsymbol{B}^{(6)} &= [
        1,0,0,0,0,0,0
    ]^\top,\qquad
    \hat{\boldsymbol{B}}^{(6)} = [
        6,-15,20,-15,6,-1
    ]^\top
\end{aligned}
\end{equation*}
together with the multiplier coefficient vector $\boldsymbol{\mu}=\tfrac{1}{6}[8,-2,-1,3,-3,1]^\top$ satisfies \eqref{FP-Multiplier} with $\alpha=1/6$ and $\beta=1/6$. Hence, by \Cref{thm:main}, the scheme admits the modified energy dissipation. The corresponding matrices $G_a$ and $G_B$ in the modified energy \eqref{eq:EG}, with $\hat{c}_i$ given by \eqref{eq:tilde-c}, satisfy \eqref{eq:pdc} and are given by
{\footnotesize
\begin{equation}
\label{eq:GaGb-bdf6}
G_a = \frac{1}{100}
\begin{bmatrix}
 286 & -273 &  229 & -125 &  17 \\
-273 &  354 & -279 &  146 & -27 \\
 229 & -279 &  260 & -136 &  25 \\
-125 &  146 & -136 &   87 & -15 \\
  17 &  -27 &   25 &  -15 &  12
\end{bmatrix},
\ 
G_B = \frac{1}{12}
\begin{bmatrix}
 13 & -6 &  0 &  4 & -5 &  2 \\
 -6 &  7 & -4 &  0 &  2 & -1 \\
  0 & -4 &  6 & -4 &  1 &  0 \\
  4 &  0 & -4 &  6 & -4 &  1 \\
 -5 &  2 &  1 & -4 &  5 & -2 \\
  2 & -1 &  0 &  1 & -2 &  1
\end{bmatrix}.
\end{equation}
}
In particular, \eqref{eq:tilde-c} gives $\hat{c}_0=625/3$, which appears in the time-step bound \eqref{eq:tau-max}.
\end{proposition}

\begin{proof}
The polynomial $M(z;\boldsymbol{\mu})=\tfrac{1}{6}(8z^5-2z^4-z^3+3z^2-3z+1)$ is Schur stable according to the Schur--Cohn criterion. The pairs $(M(z;\boldsymbol{a}^{(6)}),M(z;\boldsymbol{\mu}))$ and $(M(z;\boldsymbol{B}^{(6)}), M(z;\boldsymbol{\nu}))$ are both coprime, as the corresponding resultants are nonzero. We omit the detailed calculations. The stability and coprimality conditions can also be verified from the roots of the associated polynomials.

The polynomial $P_a$ in \eqref{eq:FP-poly} is
\begin{equation*}
    P_a^{(6)}(x) = \frac{1}{6} + \frac{268x^5 - 276x^4 - 436x^3 + 511x^2 + 5x + 3}{90}.
\end{equation*}
By Sturm's theorem, $P_a^{(6)}(x) - 1/6$ has no real roots in $[-1,1]$. Together with $P_a^{(6)}(1) = 1 > 1/6$, we obtain $P_a^{(6)}(x) \ge 1/6$ for all $x\in [-1,1]$. Analogously, the polynomial $P_B$ in \eqref{eq:FP-poly} is
\begin{equation*}
    P_B^{(6)}(x) = \frac{1}{6} + \frac{8x^{2}(1-x)^{2}(1+x)}{3}.
\end{equation*}
It follows directly that $P_B^{(6)}(x)\geq 1/6$ for all $x\in[-1,1]$. Hence \eqref{eq:FP-poly} holds with $\alpha=1/6$ and $\beta=1/6$.

Therefore, the multiplier $\boldsymbol{\mu}=\tfrac{1}{6}[8,-2,-1,3,-3,1]^\top$ is indeed a feasible solution to \eqref{FP-Multiplier} for IMEX-BDF6 with $\alpha=1/6$ and $\beta=1/6$. Hence, by \Cref{thm:main}, the IMEX-BDF6 scheme preserves the modified energy dissipation. The matrices in \eqref{eq:GaGb-bdf6} are obtained from \cref{rem:determine-G} with the polynomial pairs used in \Cref{thm:main}. Their positive definiteness and the inequalities in \eqref{eq:pdc} can be verified directly.
\end{proof}

\begin{proposition}
\label{prop:wsbdf7-feasible}
The IMEX-WSBDF7 scheme of \cite{2025IMA-NO-Akrivis-BDF7}, defined by
\begin{equation}
\label{eq:A-B-Bhat-wsbdf7}
\begin{aligned}
\boldsymbol{A}^{(7)} &= \Bigl[
    \frac{1049}{140},-\frac{239}{10},\frac{75}{2},-40,
        \frac{355}{12},-\frac{141}{10},\frac{39}{10},-\frac{10}{21}
\Bigr]^\top,\\
\boldsymbol{B}^{(7)} &= \bigl[
    3,-2,0,0,0,0,0,0
\bigr]^\top,\\
\hat{\boldsymbol{B}}^{(7)} &= \bigl[
    19,-63,105,-105,63,-21,3
\bigr]^\top
\end{aligned}
\end{equation}
together with the multiplier coefficient vector
$\boldsymbol{\mu} = \frac{1}{5}[20,-21,3,12,-16,11,-4]^\top$
satisfies \eqref{FP-Multiplier} with $\alpha=1/4$ and $\beta=1/100$.
By \Cref{thm:main}, the scheme admits modified energy dissipation. The corresponding matrices $G_a$ and $G_B$ in the modified energy \eqref{eq:EG}, with $\hat{c}_i$ given by \eqref{eq:tilde-c}, satisfy \eqref{eq:pdc} and are given by
{
\footnotesize
\begin{equation}
\label{eq:GaGb-wsbdf7-robust}
\begin{aligned}
G_a = \frac{1}{100} &\,
\begin{bmatrix}
 2544 & -4320 &  4139 & -2754 &  1180 &  -280 \\
-4320 &  8677 & -8828 &  5955 & -2694 &   684 \\
 4139 & -8828 &  9693 & -6738 &  3061 &  -774 \\
-2754 &  5955 & -6738 &  5054 & -2459 &   660 \\
 1180 & -2694 &  3061 & -2459 &  1432 &  -433 \\
 -280 &   684 &  -774 &   660 &  -433 &   164
\end{bmatrix},
\\[4pt]
G_B = \frac{1}{10000} &\,
\begin{bmatrix}
 103511 & -129612 &  43554 &  38296 & -69378 &  50646 & -17017 \\
-129612 &  198687 & -118031 &   1837 &  63785 & -59721 &  22055 \\
  43554 & -118031 &  146214 & -97072 &  28484 &   8221 &  -8370 \\
  38296 &    1837 & -97072 & 137157 & -106715 &  50284 & -11787 \\
 -69378 &   63785 &  28484 & -106715 & 121949 & -77349 &  23224 \\
  50646 &  -59721 &   8221 &  50284 & -77349 &  60692 & -21773 \\
 -17017 &   22055 &  -8370 & -11787 &  23224 & -21773 &   9668
\end{bmatrix}.
\end{aligned}
\end{equation}
}
In particular, \eqref{eq:tilde-c} gives $\hat{c}_0=33856/5$, which appears in the time-step bound \eqref{eq:tau-max}.
\end{proposition}

\begin{proof}
The polynomial $M(z;\boldsymbol{\mu}) = \frac{1}{5}(20z^6 - 21z^5 + 3z^4 + 12z^3 - 16z^2 + 11z - 4)$ is Schur stable according to the Schur--Cohn criterion. The pairs $(M(z;\boldsymbol{a}^{(7)}),M(z;\boldsymbol{\mu}))$ and $(M(z;\boldsymbol{B}^{(7)}), M(z;\boldsymbol{\nu}))$ are both coprime, as the corresponding resultants are nonzero. We omit the detailed calculations. These stability and coprimality conditions can also be verified from the roots of the associated polynomials. Using \eqref{def:coefs-a-b-c}, we obtain
\begin{equation*}
\boldsymbol{a}^{(7)} = \Big[\frac{1049}{140},-\frac{2297}{140},\frac{2953}{140},
   -\frac{2647}{140},\frac{1121}{105},-\frac{719}{210},\frac{10}{21}\Big]^\top.
\end{equation*}

The polynomial $P_a$ in \eqref{eq:FP-poly} is
\begin{equation*}
    P_a^{(7)}(x) = \tfrac{1}{4} + \tfrac{-274816x^6 + 467536x^5 + 83336x^4 - 495080x^3 + 207614x^2 + 12114x + 871}{2100}.
\end{equation*}
By Sturm's theorem, $P_a^{(7)}(x) - 1/4$ has no real roots in $[-1,1]$. Together with $P_a^{(7)}(1)-1/4 > 0$, we obtain $P_a^{(7)}(x)\geq 1/4$ for all $x\in[-1,1]$.
Analogously, the polynomial $P_B$ in \eqref{eq:FP-poly} is
\begin{equation*}
    P_B^{(7)}(x) = \tfrac{1}{100} + \tfrac{9 + 20(1-x)Q_B^{(7)}(x)}{100}, \quad Q_B^{(7)}(x) = 384x^5-80x^4-328x^3+144x^2+8x+1.
\end{equation*}
By Sturm's theorem, the polynomial $Q_B^{(7)}(x)$ has no real roots in $[-1,1]$. Together with $Q_B^{(7)}(1)>0$, it is strictly positive on $[-1,1]$. Combined with $1-x\ge 0$ on $[-1,1]$, this gives $P_B^{(7)}(x)\geq 1/100$ for all $x\in[-1,1]$. Hence \eqref{eq:FP-poly} holds with $\alpha=1/4$ and $\beta=1/100$.

Therefore, the multiplier $\boldsymbol{\mu} = \frac{1}{5}[20,-21,3,12,-16,11,-4]^\top$ is indeed a feasible solution to \eqref{FP-Multiplier} for IMEX-WSBDF7 with $\alpha=1/4$ and $\beta=1/100$. Hence, by \Cref{thm:main}, the IMEX-WSBDF7 scheme preserves the modified energy dissipation. The matrices in \eqref{eq:GaGb-wsbdf7-robust} are obtained from \cref{rem:determine-G} with the polynomial pairs used in \Cref{thm:main}. Their positive definiteness and the inequalities in \eqref{eq:pdc} can be verified directly.
\end{proof}

\subsection{An energy-dissipative IMEX-LMM8}

We now construct an eighth-order energy-dissipative IMEX-LMM scheme by searching over a set of simple integer choices of the implicit coefficient vector $\boldsymbol B^{(8)}$:
\[
   \left\{ \boldsymbol B^{(8)} \mid B_j^{(8)}\in\mathbb Z,\; |B_j^{(8)}|\le 2,\;
    B_j^{(8)}=0\ \text{for } j>4,\;
    \sum_{j=0}^{8}B_j^{(8)}=1\right\}.
\]
%are enumerated.
% For each candidate $\boldsymbol B^{(8)}$, the remaining scheme coefficients are reconstructed from the order conditions by \eqref{eq:matrix-parameterization}.
For each $\boldsymbol B^{(8)}$ in this set, we solve LP \eqref{eq:LP-multiplier} to see if there exists a multiplier candidate. If so, the corresponding scheme and this multiplier candidate are further verified against the original feasibility conditions in \eqref{FP}.

\begin{proposition}
\label{prop:lmm8-feasible}
The IMEX-LMM8 scheme defined by
\begin{equation}
\label{eq:A-B-Bhat-8th}
\begin{aligned}
\boldsymbol{A}^{(8)} &= \Bigl[
    \frac{4369}{840}, -\frac{581}{30}, \frac{181}{5}, -\frac{1327}{30}, \frac{115}{3}, -\frac{241}{10}, \frac{51}{5}, -\frac{527}{210}, \frac{11}{40}
\Bigr]^\top, \\
\boldsymbol{B}^{(8)} &= \bigl[
    2, -2, 0, 2, -1, 0, 0, 0, 0
\bigr]^\top, \\
\hat{\boldsymbol{B}}^{(8)} &= \bigl[
    14, -56, 114, -141, 112, -56, 16, -2
\bigr]^\top
\end{aligned}
\end{equation}
together with the multiplier coefficient vector $\boldsymbol{\mu} = \frac{1}{4}[9, -13, 10, 0, -5, 6, -4, 1]^\top$ satisfies \eqref{FP} with $\alpha = 1/64$ and $\beta = 1/53$. Hence, by \Cref{thm:main}, the scheme admits modified energy dissipation. The corresponding matrices $G_a$ and $G_B$ in the modified energy \eqref{eq:EG}, with $\hat{c}_i$ given by \eqref{eq:tilde-c}, satisfy \eqref{eq:pdc} and are given by
{\footnotesize
\begin{equation}
\label{eq:GaGb-8th-robust}
\begin{aligned}
% G_a = \tfrac{1}{10000} &\,
% \begin{bmatrix}
%   75508 & -160956 &  203546 & -159950 &   92595 &  -35176 &   5454 \\
% -160956 &  368482 & -475338 &  383138 & -219714 &   84760 & -13766 \\
%  203546 & -475338 &  628068 & -513809 &  296772 & -113163 &  19825 \\
% -159950 &  383138 & -513809 &  431646 & -251786 &   96567 & -16157 \\
%   92595 & -219714 &  296772 & -251786 &  152167 &  -59983 &  11044 \\
%  -35176 &   84760 & -113163 &   96567 &  -59983 &   26738 &  -5601 \\
%    5454 &  -13766 &   19825 & -16157 &   11044 &   -5601 &   2960
% \end{bmatrix},
% \\[4pt]
% G_B = \tfrac{1}{100000} &\,
% \begin{bmatrix}
%  403371 & -657208 &  461064 &   36464 & -260305 &  251117 & -163689 &  41686 \\
% -657208 & 1197243 & -931714 &    -666 &  496571 & -474187 &  275524 & -68063 \\
%  461064 & -931714 &  897993 & -194792 & -349955 &  440110 & -254248 &  56542 \\
%   36464 &    -666 & -194792 &  304571 & -153239 &  -51486 &   79210 & -20062 \\
% -260305 &  496571 & -349955 & -153239 &  404839 & -290673 &  105937 & -15675 \\
%  251117 & -474187 &  440110 &  -51486 & -290673 &  379082 & -232868 &  53905 \\
% -163689 &  275524 & -254248 &   79210 &  105937 & -232868 &  199388 & -59254 \\
%   41686 &  -68063 &   56542 &  -20062 &  -15675 &   53905 &  -59254 &  23421
% \end{bmatrix}.
G_a = {}&\,
\begin{bmatrix}
 \tfrac{75508}{10000} & -\tfrac{160956}{10000} &  \tfrac{203546}{10000} & -\tfrac{159950}{10000} &  \tfrac{92595}{10000} & -\tfrac{35176}{10000} &  \tfrac{5454}{10000} \\[3pt]
-\tfrac{160956}{10000} &  \tfrac{368482}{10000} & -\tfrac{475338}{10000} &  \tfrac{383138}{10000} & -\tfrac{219714}{10000} &  \tfrac{84760}{10000} & -\tfrac{13766}{10000} \\[3pt]
 \tfrac{203546}{10000} & -\tfrac{475338}{10000} &  \tfrac{628068}{10000} & -\tfrac{513809}{10000} &  \tfrac{296772}{10000} & -\tfrac{113163}{10000} &  \tfrac{19825}{10000} \\[3pt]
-\tfrac{159950}{10000} &  \tfrac{383138}{10000} & -\tfrac{513809}{10000} &  \tfrac{431646}{10000} & -\tfrac{251786}{10000} &  \tfrac{96567}{10000} & -\tfrac{16157}{10000} \\[3pt]
 \tfrac{92595}{10000} & -\tfrac{219714}{10000} &  \tfrac{296772}{10000} & -\tfrac{251786}{10000} &  \tfrac{152167}{10000} & -\tfrac{59983}{10000} &  \tfrac{11044}{10000} \\[3pt]
-\tfrac{35176}{10000} &  \tfrac{84760}{10000} & -\tfrac{113163}{10000} &  \tfrac{96567}{10000} & -\tfrac{59983}{10000} &  \tfrac{26738}{10000} & -\tfrac{5601}{10000} \\[3pt]
 \tfrac{5454}{10000} & -\tfrac{13766}{10000} &  \tfrac{19825}{10000} & -\tfrac{16157}{10000} &  \tfrac{11044}{10000} & -\tfrac{5601}{10000} &  \tfrac{2960}{10000}
\end{bmatrix},
\\[4pt]
G_B = {}&\,
\begin{bmatrix}
 \tfrac{403371}{100000} & -\tfrac{657208}{100000} &  \tfrac{461064}{100000} &  \tfrac{36464}{100000} & -\tfrac{260305}{100000} &  \tfrac{251117}{100000} & -\tfrac{163689}{100000} &  \tfrac{41686}{100000} \\[3pt]
-\tfrac{657208}{100000} &  \tfrac{1197243}{100000} & -\tfrac{931714}{100000} & -\tfrac{666}{100000} &  \tfrac{496571}{100000} & -\tfrac{474187}{100000} &  \tfrac{275524}{100000} & -\tfrac{68063}{100000} \\[3pt]
 \tfrac{461064}{100000} & -\tfrac{931714}{100000} &  \tfrac{897993}{100000} & -\tfrac{194792}{100000} & -\tfrac{349955}{100000} &  \tfrac{440110}{100000} & -\tfrac{254248}{100000} &  \tfrac{56542}{100000} \\[3pt]
 \tfrac{36464}{100000} & -\tfrac{666}{100000} & -\tfrac{194792}{100000} &  \tfrac{304571}{100000} & -\tfrac{153239}{100000} & -\tfrac{51486}{100000} &  \tfrac{79210}{100000} & -\tfrac{20062}{100000} \\[3pt]
-\tfrac{260305}{100000} &  \tfrac{496571}{100000} & -\tfrac{349955}{100000} & -\tfrac{153239}{100000} &  \tfrac{404839}{100000} & -\tfrac{290673}{100000} &  \tfrac{105937}{100000} & -\tfrac{15675}{100000} \\[3pt]
 \tfrac{251117}{100000} & -\tfrac{474187}{100000} &  \tfrac{440110}{100000} & -\tfrac{51486}{100000} & -\tfrac{290673}{100000} &  \tfrac{379082}{100000} & -\tfrac{232868}{100000} &  \tfrac{53905}{100000} \\[3pt]
-\tfrac{163689}{100000} &  \tfrac{275524}{100000} & -\tfrac{254248}{100000} &  \tfrac{79210}{100000} &  \tfrac{105937}{100000} & -\tfrac{232868}{100000} &  \tfrac{199388}{100000} & -\tfrac{59254}{100000} \\[3pt]
 \tfrac{41686}{100000} & -\tfrac{68063}{100000} &  \tfrac{56542}{100000} & -\tfrac{20062}{100000} & -\tfrac{15675}{100000} &  \tfrac{53905}{100000} & -\tfrac{59254}{100000} &  \tfrac{23421}{100000}
\end{bmatrix}.
\end{aligned}
\end{equation}
}
In particular, \eqref{eq:tilde-c} gives $\hat{c}_0=25043/4$, which appears in the time-step bound \eqref{eq:tau-max}.
\end{proposition}

\begin{proof}
The polynomial $M(z;\boldsymbol{\mu})=\tfrac{1}{4}(9z^7-13z^6+10z^5-5z^3+6z^2-4z+1)$ is Schur stable according to the Schur--Cohn criterion. The pairs $(M(z;\boldsymbol{a}^{(8)}),M(z;\boldsymbol{\mu}))$ and $(M(z;\boldsymbol{B}^{(8)}), M(z;\boldsymbol{\nu}))$ are both coprime, as the corresponding resultants are nonzero. We omit the detailed calculations. These stability and coprimality conditions can also be verified from the roots of the associated polynomials. Using \eqref{def:coefs-a-b-c}, we obtain
\begin{equation*}
\boldsymbol{a}^{(8)} = \Bigl[
    \frac{4369}{840}, -\frac{11899}{840}, \frac{18509}{840}, -\frac{18647}{840}, \frac{13553}{840}, -\frac{6691}{840}, \frac{1877}{840}, -\frac{11}{40}
\Bigr]^\top.
\end{equation*}

The polynomial $P_a$ in \eqref{eq:FP-poly} is
\begin{equation*}
    P_a^{(8)}(x) = \tfrac{1}{64} + \tfrac{293120x^7 - 606656x^6 - 340512x^5 + 1578912x^4 - 1167696x^3 + 238272x^2 + 6032x + 5143}{6720}.
\end{equation*}
By Sturm's theorem, $P_a^{(8)}(x) - 1/64$ has no real roots in $[-1,1]$. Together with $P_a^{(8)}(1) = 1 > 1/64$, we obtain $P_a^{(8)}(x)\geq 1/64$ for all $x\in[-1,1]$. 
Analogously, the polynomial $P_B$ in \eqref{eq:FP-poly} is
\begin{equation*}
    P_B^{(8)}(x) = \tfrac{1}{53} + \tfrac{6784x^7 - 13568x^6 - 1696x^5 + 16960x^4 - 6360x^3 - 3498x^2 + 1113x + 367}{212}.
\end{equation*}
By Sturm's theorem, $P_B^{(8)}(x)-1/53$ has no real roots in $[-1,1]$. Together with $P_B^{(8)}(1)-1/53>0$, we obtain $P_B^{(8)}(x)\geq 1/53$ for all $x\in[-1,1]$. Hence \eqref{eq:FP-poly} holds with $\alpha=1/64$ and $\beta=1/53$.

 Therefore, the IMEX-LMM8 scheme and multiplier above give a feasible solution to \eqref{FP} with $\alpha=1/64$ and $\beta=1/53$. Hence, by \Cref{thm:main}, the IMEX-LMM8 scheme preserves the modified energy dissipation. As before, the matrices in \eqref{eq:GaGb-8th-robust} are obtained from \cref{rem:determine-G} with the polynomial pairs used in \Cref{thm:main}, and their positive definiteness and the inequalities in \eqref{eq:pdc} can be verified directly.
\end{proof}

\begin{remark}
We mention that the values of $\alpha$ and $\beta$ used in \cref{prop:bdf6-feasible,prop:wsbdf7-feasible,prop:lmm8-feasible} might not be optimal.
\end{remark}

\begin{remark}[$H^1$-stability of LMMs]\label{rem:linear-parabolic}
The generalized Dahlquist's theory (\Cref{thm:equivalence}) can also be used to establish the $H^1$-stability analysis of LMMs for linear parabolic equations.
Consider, for example, a homogeneous linear parabolic problem discretized by the $k$-step LMM
\[
    \sum_{i=0}^k A_i^{(k)}u^{n+1-i}
    +
    \tau\mathcal L\sum_{i=0}^k B_i^{(k)}u^{n+1-i}
    =0,
\]
where $\mathcal L$ is a linear, self-adjoint, and positive semi-definite operator. This can be seen as a special case of gradient flow with $\mathcal M = -\mathcal I$ and $f(u) = 0$.
%If the coefficients of the scheme and the multiplier satisfy \eqref{FP}, then the same algebraic argument as in \cref{thm:main} gives positive definite matrices $G_a$ and $G_B$. Testing the LMM with $\sum_{i=0}^k\nu_i u^{n+1-i}$ yields 
Consequently, if the feasibility problem \eqref{FP} is solvable,
we directly obtain the following decreasing functional:
\[
    (\boldsymbol v_n,\boldsymbol v_n)_{G_a}
    +
    \tau(\boldsymbol u_n,\mathcal L\boldsymbol u_n)_{G_B}.
\]
In the case of $\mathcal L=-\Delta$, this decreasing functional gives an $H^1$-seminorm stability estimate.
%This means that the feasible multipliers proposed in \Cref{prop:bdf6-feasible,prop:wsbdf7-feasible,prop:lmm8-feasible} yield the $H^1$-seminorm stability of the corresponding LMMs for the homogeneous linear parabolic problem.
% The examples in \cref{prop:bdf6-feasible,prop:wsbdf7-feasible,prop:lmm8-feasible} therefore also fall under this $H^1$-seminorm stability argument.

% This is not an $L^2$ G-stability statement.  Indeed,
% \[
%     M(z;\boldsymbol A^{(k)})=(z-1)M(z;\boldsymbol a^{(k)}),
%     \qquad
%     M(z;\boldsymbol\nu_{H^1})=(z-1)M(z;\boldsymbol\mu),
% \]
% so the pair $M(z;\boldsymbol A^{(k)})$ and
% $M(z;\boldsymbol\nu_{H^1})$ has the common factor $z-1$, and the
% corresponding state matrix is only positive semidefinite.

For the $L^2$-stability, we shall find a vector $\tilde{\boldsymbol\nu}=[\tilde{\nu}_0,\ldots,\tilde{\nu}_k]^\top$ satisfying
\[
    \sum_{i=0}^k\tilde{\nu}_i=1,
    \qquad
    \deg M(z;\tilde{\boldsymbol\nu})=k,
\]
such that the generating polynomial $M(z;\tilde{\boldsymbol\nu})$ is Schur stable and is coprime to both $M(z;\boldsymbol A^{(k)})$ and $M(z;\boldsymbol B^{(k)})$, and there exist $\tilde\alpha, \tilde{\beta}>0$ such that
\[
\begin{aligned}
    \operatorname{Re}\bigl\{
    M(\mathrm e^{\mathrm i\theta};\boldsymbol A^{(k)})
    M(\mathrm e^{-\mathrm i\theta};\tilde{\boldsymbol\nu})
    \bigr\}&
    \ge
    2\tilde{\alpha}(1-\cos\theta), \\
    \operatorname{Re}\bigl\{
    M(\mathrm e^{\mathrm i\theta};\boldsymbol B^{(k)})
    M(\mathrm e^{-\mathrm i\theta};\tilde{\boldsymbol\nu})
    \bigr\}&
    \ge
    \tilde{\beta},\qquad 0\le\theta\le\pi.
\end{aligned}
\]
If such $\tilde{\boldsymbol\nu}$ exists, \Cref{thm:equivalence}(v) gives two positive definite matrices $\tilde{G}_{a}$ and $\tilde{G}_{B}$. Testing the LMM with $\sum_{i=0}^k\tilde{\nu}_i u^{n+1-i}$ then gives the following decreasing functional
\[
    (\boldsymbol u_n,\boldsymbol u_n)_{\tilde{G}_{a}}
    +
    \tau(\boldsymbol u_n,\mathcal L\boldsymbol u_n)_{\tilde{G}_{B}}.
\]
The $L^2$-stability then follows.

Take the BDF6 scheme as an example, whose $L^2$- and $H^1$-stability has been well-studied in \cite{2021SINUM-NO-Akrivis-BDF6}.
From the aforementioned discussions, we can derive the $H^1$-seminorm stability straightforwardly from the multiplier ${\boldsymbol\mu}$ in Proposition \ref{prop:bdf6-feasible}. Furthermore, following a similar procedure in \cref{subsec:parameterization}, we can obtain the following admissible $\tilde{\boldsymbol\nu}$ and the corresponding $\tilde{\alpha}$ and $\tilde{\beta}$:
\[
    \tilde{\boldsymbol\nu}
    =
    \frac12[7,-7,3,0,-2,0,1]^\top,
    \quad
    \tilde{\alpha}=\frac16,\quad \tilde{\beta}=\frac13 .
\]
The $L^2$-stability then follows.
Compared with the BDF6's stability argument in \cite{2021SINUM-NO-Akrivis-BDF6}, the multiplier there is used to establish the quadratic decomposition of only one coefficient polynomial, and the other polynomial is treated after summing over the time levels.
The multiplier $\tilde{\boldsymbol\nu}$ obtained from our framework gives the quadratic decomposition of both two polynomials corresponding to $\boldsymbol A^{(6)}$ and $\boldsymbol B^{(6)}$.
\end{remark}

% \begin{remark}\label{rem:linear-parabolic}
% The generalized \textcolor{red}{Dahlquist} equivalence theorem (\Cref{thm:equivalence}) is purely algebraic and can also be used in multiplier-based stability arguments for linear parabolic equations.
% For example, consider a homogeneous linear parabolic equation discretized by a $k$-step LMM
% \[
%     \sum_{i=0}^k A_i^{(k)} u^{n+1-i} = \tau \mathcal{M} \mathcal{L} \sum_{i=0}^k B_i^{(k)} u^{n+1-i}
% \]
% tested with the multiplier $\sum_{i=0}^{k}\nu_i u^{n+1-i}$.
% The relevant unit-circle conditions are
% \[
% \operatorname{Re}\bigl\{M(z;\boldsymbol{A}^{(k)})
% \,\overline{M(z;\boldsymbol{\nu})}\bigr\}\ge 0,
% \qquad
% \operatorname{Re}\bigl\{M(z;\boldsymbol{B}^{(k)})
% \,\overline{M(z;\boldsymbol{\nu})}\bigr\}\ge 0,
% \quad |z|=1.
% \]
% The first follows from \eqref{eq:FP-D1} after multiplication by the nonnegative factor $ |z-1|^2 $, and the second follows directly from \eqref{eq:FP-D2}.
% Consequently, any scheme-multiplier pair satisfying \eqref{FP} provides the algebraic conditions for $L^2$-stability of homogeneous linear parabolic problems. Thus the multipliers displayed in \Cref{prop:bdf6-feasible,prop:wsbdf7-feasible,prop:lmm8-feasible} can also be used for the corresponding $L^2$-stability arguments.
% \end{remark}

\section{Numerical experiments}\label{sec:numerics}

This section presents two numerical experiments to verify the temporal accuracy and modified energy dissipation of the IMEX-LMMs studied in \cref{sec:feasibility}. All computations use a Fourier pseudo-spectral discretization on a $256\times256$ mesh \cite{shen2011spectral}.

\subsection{Temporal convergence test}

To verify the temporal convergence order, we use the manufactured solution
\[
u(x,y,t)
=
0.15\,\mathrm{e}^{-t}\Bigl[\cos x \cos y+\frac12\cos(2x)\Bigr],
\]
for the PFC equation on $(0,2\pi)^2$ with $\varepsilon=0.025$. An appropriate source term is added, using the implicit coefficient vector $\boldsymbol{B}^{(k)}$, so that $u$ is the exact solution. The starting values $u^0, \ldots, u^{k-1}$ are taken from the exact solution to isolate the time discretization error. We integrate to the final time $T=1$ with time steps $\tau=T/N$,  where $N$ is taken from $\{20,25,32,40,50,64\}$ so that $\tau$ has a terminating decimal representation. The error is measured by the discrete $L^2$ norm
\[
e_{L^2}
=\biggl(h_xh_y\sum_{i,j}\left|u_\tau(x_i,y_j,T)-u(x_i,y_j,T)\right|^2\biggr)^{1/2},
\]
and the convergence rate is computed from two consecutive time steps.
\Cref{tab:conv-678} shows that the observed rates agree with the expected orders of 6, 7, and 8.

\begin{table}[htb!]
\centering
\setlength{\extrarowheight}{2pt}
\begin{tabular}{c|cc|cc|cc}
\toprule
\multirow{2}{*}{$\tau$} & \multicolumn{2}{c|}{IMEX-BDF6} & \multicolumn{2}{c|}{IMEX-WSBDF7} & \multicolumn{2}{c}{IMEX-LMM8} \\
\cline{2-7}
& $e_{L^2}$ & rate & $e_{L^2}$ & rate & $e_{L^2}$ & rate \\
\midrule
1/20 & 1.8557e-8 & -- & 8.7777e-9 & -- & 9.6313e-10 & -- \\
1/25 & 4.9686e-9 & 5.91 & 1.9096e-9 & 6.84 & 1.3546e-10 & 8.79 \\
1/32 & 1.1393e-9 & 5.97 & 3.4501e-10 & 6.93 & 1.8681e-11 & 8.03 \\
1/40 & 2.9967e-10 & 5.98 & 7.2118e-11 & 7.01 & 3.3021e-12 & 7.77 \\
1/50 & 7.8649e-11 & 5.99 & 1.5045e-11 & 7.02 & 5.8202e-13 & 7.78 \\
1/64 & 1.7892e-11 & 6.00 & 2.6558e-12 & 7.03 & 6.9336e-14 & 8.62 \\
\bottomrule
\end{tabular}
\caption{$L^2$ errors and convergence rates at $T=1$ ($\varepsilon = 0.025$).}
\label{tab:conv-678}
\end{table}

\subsection{Modified energy dissipation and crystal growth}

We consider the PFC equation without a source term in the periodic domain $(0,32)^2$ with $\varepsilon=0.25$. The initial condition is given by
\[
\begin{aligned}
    &\,u^0(x,y)
    = 
    0.07
    -0.02 \cos \Bigl(\frac{2\pi(x-12)}{32}\Bigr)\sin\Bigl(\frac{2\pi(y-1)}{32}\Bigr) \\
    &\qquad +0.02\cos\Bigl(\frac{2\pi(x+10)}{32}\Bigr)\cos\Bigl(\frac{2\pi(y+3)}{32}\Bigr)
    -0.01\sin^2\Bigl(\frac{2\pi x}{32}\Bigr)\sin^2\Bigl(\frac{4\pi(y-6)}{32}\Bigr).
\end{aligned}
\]
The standard PFC free energy is
\[
E_{\mathrm{PFC}}[u] = \int_\Omega \kbraB{\frac12 |(1+\Delta)u|^2 + \frac14(u^2-\varepsilon)^2}\,\mathrm{d}\boldsymbol{x}.
\]
In the abstract framework, we use the shifted splitting
\[
\mathcal L=(I+\Delta)^2+\varepsilon I,\qquad
f(u)=u^3-2\varepsilon u,\qquad
F(u)=\frac14\bigl(u^2-2\varepsilon\bigr)^2.
\]
The associated energy reads
\[
E[u]=E_{\mathrm{PFC}}[u]+\frac{3\varepsilon^2}{4}|\Omega|,
\]
which differs from $E_{\mathrm{PFC}}$ only by a constant and does not affect energy dissipation. For this test, the constant is $C_0 = 48$. For comparison with the standard PFC energy, we show the shifted modified energy $E_G^n-C_0$. This subtraction of $C_0$ does not affect dissipation and places the modified energy on the same scale as $E_{\mathrm{PFC}}$.

We take $\tau=0.001$ and integrate the three schemes to $T=150$. The starting values $u^0,\ldots,u^{k-1}$ are generated by the four-stage eighth-order Gauss--Legendre Runge--Kutta method \cite{book-solving-ODE2}, and the resulting fully implicit nonlinear stage equations are solved by fixed-point iteration with tolerance $10^{-12}$. As a reference curve, we also compute the standard PFC energy $E_{\mathrm{PFC}}^n$ by IMEX-LMM8 with step size $\tau_{\mathrm{ref}}=10^{-4}$.
For the cubic nonlinearity in this test, we use the artificial truncation bound $R_\ast=2$. The computed solutions are checked a posteriori to satisfy
\[
\max_{k=6,7,8}\max_{0\le n\le N}\|u_k^n\|_{L^\infty}
\le 0.664 < R_\ast ,
\]
so that the truncation is inactive along the numerical trajectories. The corresponding Lipschitz constant used in the modified energy is
\[
\ell_f=\max_{|s|\le R_\ast}|3s^2-2\varepsilon|
      =\max\{2\varepsilon,\,|3R_\ast^2-2\varepsilon|\}
      =\frac{23}{2},
\]
which is used for all schemes when evaluating $E_G^n$.

\Cref{fig:energy-crystal-compact}(a) shows that the shifted modified energies $E_G^n-C_0$ decay for all three schemes and remain close to the reference PFC energy curve. \Cref{fig:energy-crystal-compact}(b) shows representative snapshots computed by IMEX-LMM8. The solution evolves from short-wavelength stripes to a coarser lamellar structure, consistent with the energy decay observed in \cref{fig:energy-crystal-compact}(a).

\begin{figure}[htbp]
    \centering
    \begin{subfigure}[b]{0.35\textwidth}
        \centering
        \includegraphics[width=\textwidth]{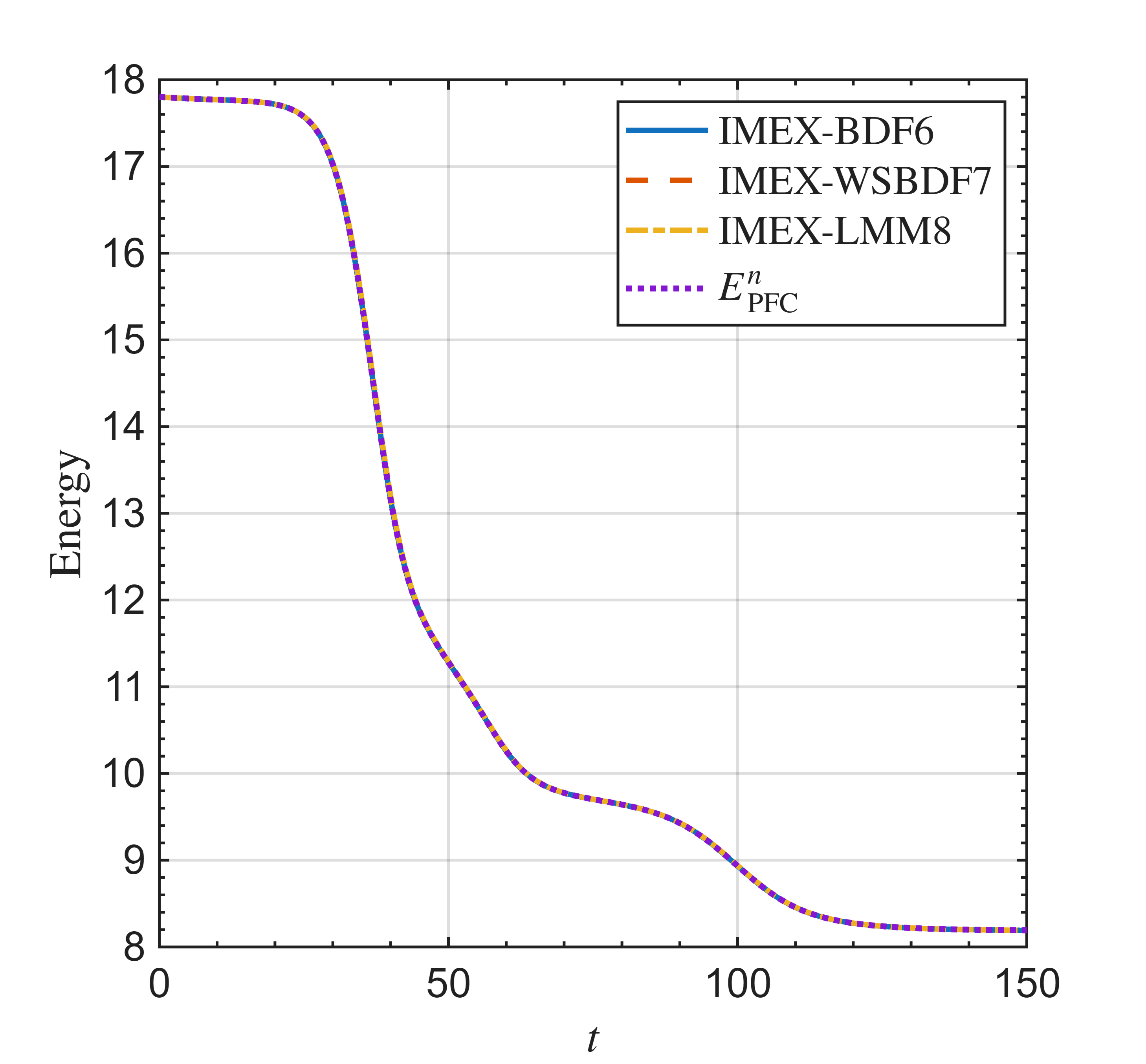}
        \caption{Shifted modified energy.}
        % \label{fig:energy-678-compact}
    \end{subfigure}
    \hfill
    \begin{subfigure}[b]{0.64\textwidth}
        \centering
        \includegraphics[width=\textwidth]{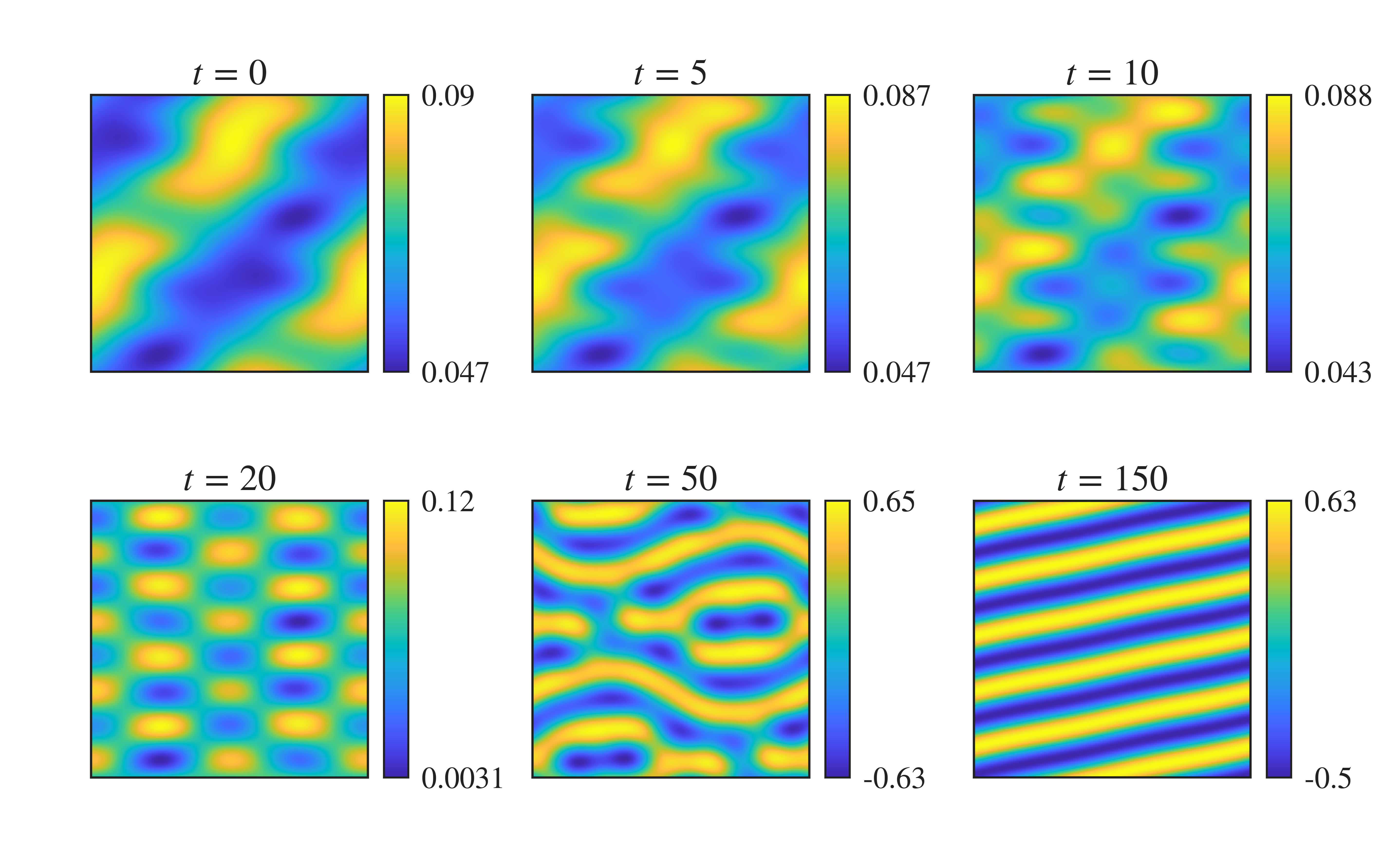}
        \caption{Crystal-growth snapshots.}
        % \label{fig:crystal-lmm8-compact}
    \end{subfigure}
    \caption{(a) Shifted modified energies $E_G^n-C_0$ of the three schemes ($\tau=0.001$), together with the standard PFC energy computed by IMEX-LMM8 ($\tau_{\mathrm{ref}}=10^{-4}$). (b) Crystal-growth snapshots produced by IMEX-LMM8 ($\tau=0.001$). The color scale is adjusted independently in each snapshot.}
    \label{fig:energy-crystal-compact}
\end{figure}

Although the cubic nonlinearity in the PFC equation is not globally Lipschitz, the global-in-time arguments in \cite{WangZhaoLiao-2025-SINUM, global-in-time-LiQiaoWangZheng} provide a possible way to establish uniform $L^\infty$-norm bounds for numerical solutions and consequently to remove the global Lipschitz assumption without introducing auxiliary truncation bounds.

\section{Conclusion}\label{sec:conclusion}

This work developed a unified framework to establish the energy dissipation of high-order IMEX-LMMs for gradient flows, based on general multipliers in the form of linear combinations of first-order differences. The classical Dahlquist's theory is generalized, which helps to ensure the energy dissipation under mild time-step restrictions through degree, positivity, Schur stability, and coprimality conditions for the associated generating polynomials. Moreover, the constructed modified energy is consistent with the original energy as the time step $\tau\to 0$.

These conditions together with the order conditions of IMEX-LMMs form a feasibility problem over the coefficients of the scheme and the multiplier. It can be solved numerically (based on the LP solver) to find suitable multipliers for the IMEX-BDF6 and IMEX-WSBDF7, and to construct new  energy-dissipative IMEX-LMMs such as IMEX-LMM8. 
Such analysis can also be used in $L^2$- and $H^1$-stability of LMMs for linear parabolic equations as stated in \cref{rem:linear-parabolic}.

Several questions remain to be explored in future work. First, the current analysis assumes that the nonlinearity is globally Lipschitz continuous. The case of locally Lipschitz continuous nonlinearity, such as the cubic term in the PFC equation, may be handled by combining our framework with global-in-time energy estimates \cite{WangZhaoLiao-2025-SINUM, global-in-time-LiQiaoWangZheng}.
Second, numerically speaking, we can not find a ninth-order energy-dissipative IMEX-LMM, whose existence is still unclear as the feasibility problem is nonlinear.
Finally, how this framework can be used to variable-step IMEX-BDF schemes \cite{ChenWangYanZhang-2019-SINUM, LiaoTangZhou-2020-SINUM} is still unknown.

\bibliographystyle{siamplain}
\bibliography{References}

\end{document}